\documentclass[reqno,a4paper]{amsart}
\usepackage[english]{babel}

\usepackage{graphicx, amsmath, amsfonts,amssymb,amsthm, amscd,amsbsy,amstext, tikz,tikz-cd, mathrsfs, mathtools,mathpazo, enumerate,enumitem, tensor, a4wide} 
\usepackage{latexsym,ifthen,stmaryrd,fancyhdr,empheq,verbatim, epigraph}
\usepackage{dynkin-diagrams}
\mathtoolsset{showonlyrefs,showmanualtags}

\usepackage{hyperref}
\hypersetup{anchorcolor=black, linktocpage=true,
colorlinks=true,
citecolor=red,
linkcolor=MyDarkBlue,
urlcolor=black,
breaklinks=true,
plainpages=true
}

\usetikzlibrary{positioning,shapes,shadows,arrows}

\usepackage{color}
\definecolor{MyDarkBlue}{rgb}{0.15,0.25,0.45}

\usepackage[hyperpageref]{backref} 

\usepackage{epigraph}

\usepackage[utf8x]{inputenc}

\usepackage[toc,page]{appendix}
\usepackage[mathscr]{euscript} %with mathscr option
\allowdisplaybreaks

\newcommand{\calA}{{\mathcal A}}
\newcommand{\calB}{{\mathcal B}}
\newcommand{\calC}{{\mathcal C}}

\newcommand{\calE}{{\mathcal E}}
\newcommand{\calF}{{\mathcal F}}
\newcommand{\calG}{{\mathcal G}}
\newcommand{\calH}{{\mathcal H}}
\newcommand{\calI}{{\mathcal I}}

\newcommand{\calK}{{\mathcal K}}
\newcommand{\calL}{{\mathcal L}}
\newcommand{\calM}{\mathcal{M}}
\newcommand{\calN}{{\mathcal N}}

\newcommand{\calP}{{\mathcal P}}
\newcommand{\calQ}{{\mathcal Q}}

\newcommand{\calS}{\mathcal{S}}
\newcommand{\calT}{{\mathcal T}}

\newcommand{\calV}{{\mathcal V}}

\newcommand{\R}{{\mathbb{R}}}
\newcommand{\N}{{\mathbb{N}}}
\newcommand{\C}{{\mathbb{C}}}
\newcommand{\Z}{{\mathbb{Z}}}
\newcommand{\Q}{{\mathbb{Q}}}

\newcommand{\PP}{{\mathbb{P}}}

\newcommand{\LL}{{\mathbb{L}}}

\newcommand{\balpha}{{\boldsymbol\alpha}}

\newcommand{\scrA}{{\mathscr A}}
\newcommand{\scrC}{{\mathscr C}}

\newcommand{\scrF}{{\mathscr F}}

\newcommand{\scrO}{{\mathscr O}}
\newcommand{\scrT}{{\mathscr T}}

\newcommand{\scrX}{\mathscr X}
\newcommand{\scrM}{\mathscr M}

\newcommand{\sfH}{\mathsf H}
\newcommand{\sfG}{\mathsf G}
\newcommand{\sfK}{{\mathsf K}}
\newcommand{\catP}{\mathsf{P}}

\newcommand{\sfS}{{\mathsf S}}

\newcommand{\sfs}{\mathsf{s}}
\newcommand{\sft}{\mathsf{t}}

\newcommand{\fraksl}{\mathfrak{sl}}
\newcommand{\frakg}{\mathfrak{g}}

\newcommand{\bfd}{\mathbf{d}}
\newcommand{\bfe}{\mathbf{e}}
\newcommand{\bfm}{\mathbf{m}}

\newcommand{\ch}{{\rm ch}}

\newcommand{\CH}{{\mathcal H}}

\newcommand{\Hom}{\mathsf{Hom}}

\newcommand{\Ext}{\mathsf{Ext}}

\newcommand{\End}{\mathsf{End}}

\newcommand{\Pic}{\mathsf{Pic}}

\newcommand{\CT}{{\mathcal T}}

\newcommand{\GL}{\mathsf{GL}}

\newcommand{\sfRep}{\mathsf{Rep}}
\newcommand{\catCoh}{\mathsf{Coh}}
\newcommand{\catMod}{\mathsf{Mod}}

\newcommand{\catCohps}{\mathsf{Coh}_{\mathsf{ps}}}
\newcommand{\catPps}{\mathsf{P}_{\mathsf{ps}}}

\newcommand{\sfch}{\mathsf{ch}}

\newcommand{\Mod}{\textrm{-} \mathsf{Mod}}

\newcommand{\catPerf}{\mathsf{Perf}}

\newcommand{\catCohb}{\mathsf{Coh}^{\mathsf{b}}}

\newcommand{\catQCoh}{\mathsf{QCoh}}
\newcommand{\catDb}{\mathsf{D}^\mathsf{b}}
\newcommand{\catDbps}{\mathsf{D}^\mathsf{b}_{\mathsf{ps}}}

\newcommand{\ev}{\mathsf{ev}}

\newcommand{\bfCoh}{\mathbf{Coh}}

\newcommand{\bfPerCoh}{\mathbf{PerCoh}}

\makeatletter
\newcommand{\ostar}{\mathbin{\mathpalette\make@circled\star}}
\newcommand{\make@circled}[2]{%
	\ooalign{$\m@th#1\smallbigcirc{#1}$\cr\hidewidth$\m@th#1#2$\hidewidth\cr}%
}
\newcommand{\smallbigcirc}[1]{%
	\vcenter{\hbox{\scalebox{0.77778}{$\m@th#1\bigcirc$}}}%
}
\makeatother

\newcommand{\triend}{\parbox{2mm}{\hfill} \hfill\mbox{\hspace{0.2mm}}\hfill$\triangle$}
\newcommand{\ocend}{\parbox{2mm}{\hfill} \hfill\mbox{\hspace{0.2mm}}\hfill$\oslash$}

\newtheorem{theorem}{Theorem}
\newtheorem{theoremintroduction}{Theorem}

\newtheorem{proposition}[theorem]{Proposition}
\newtheorem{lemma}[theorem]{Lemma}
\newtheorem{corollary}[theorem]{Corollary}

\newtheorem{corollary*}{Corollary}
\newtheorem*{theorem*}{Theorem}
\newtheorem*{proposition*}{Proposition}
\newtheorem*{conjecture*}{Conjecture}

\numberwithin{equation}{section}
\numberwithin{theorem}{section}

\theoremstyle{remark}
\newtheorem{ex}[theorem]{Example}
\newenvironment{example}{\begin{ex}}{\triend\end{ex}}

\theoremstyle{remark}
\newtheorem{rem}[theorem]{Remark}

\newenvironment{remark}{\begin{rem}}{\triend\end{rem}}

\theoremstyle{definition}
\newtheorem{defin}[theorem]{Definition}
\newenvironment{definition}{\begin{defin}}{\ocend\end{defin}}

\newtheorem{notat}[theorem]{Notation}
\newenvironment{notation}{\begin{notat}}{\ocend\end{notat}}

\title[McKay correspondence, cohomological Hall algebras and categorification]{McKay correspondence, cohomological Hall algebras and categorification} 

\author[D.-E.~Diaconescu]{Duiliu-Emanuel Diaconescu}
\address[Duiliu-Emanuel Diaconescu]{New High Energy Theory Center - Serrin Building, Rutgers, The State University Of New Jersey, 126 Frelinghuysen Rd., Piscataway, NJ 08854-8019, USA}
\curraddr{}
\email{\href{mailto:duiliu@physics.rutgers.edu}{duiliu@physics.rutgers.edu}}

\author[M.~Porta]{Mauro Porta}
\address[Mauro Porta]{Institut de recherche mathématique avancée (IRMA), Université de Strasbourg, France}
\curraddr{}
\email{\href{mailto:porta@math.unistra.fr}{porta@math.unistra.fr}}

\author[F.~Sala]{Francesco Sala}
\address[Francesco Sala]{Università di Pisa, Dipartimento di Matematica, Largo Bruno Pontecorvo 5, 56127 Pisa (PI), Italy}
\address{Kavli IPMU (WPI), UTIAS, The University of Tokyo, Kashiwa, Chiba 277-8583, Japan}
\curraddr{}
\email{\href{mailto:francesco.sala@unipi.it}{francesco.sala@unipi.it}}

\thanks{The work of D.-E.~D. was partially supported by NSF grant DMS-1802410, while the work of F.~S. was partially supported by JSPS KAKENHI Grant Numbers JP21K03197 and JP18K13402.}



\begin{document}

\begin{abstract}

Let $\pi\colon Y\to X$ denote the canonical resolution of the two dimensional Kleinian singularity $X$ of type ADE. In the present paper, we establish isomorphisms between the cohomological and K-theoretical Hall algebras of $\omega$-semistable properly supported sheaves on $Y$ with fixed slope $\mu$ and $\zeta$-semistable finite-dimensional representations of the preprojective algebra of affine type ADE of slope zero respectively, under some conditions on $\zeta$ depending on the polarization $\omega$ and $\mu$. These isomorphisms are induced by the derived McKay correspondence. In addition, they are interpreted as decategorified versions of a monoidal equivalence between the corresponding categorified Hall algebras. In the type A case, we provide a finer descriptions of the cohomological, K-theoretical and categorified Hall algebra of $\omega$-semistable properly supported sheaves on $Y$ with fixed slope $\mu$: for example, in the cohomological case, the algebra can be given in terms of Yangians of finite type ADE Dynkin diagrams.

\end{abstract}

\maketitle\thispagestyle{empty}

\tableofcontents

\section{Introduction}

Let $G$ be a finite subgroup of $\mathsf{SL}(2, \C)$ and let $\calQ$ be the affine Dynkin diagram corresponding to $G$. Let $\pi\colon Y\to \C^2/G$ denote the canonical resolution of the two dimensional Kleinian singularity associated with $G$. The present paper aims to investigate a relation between the \textit{derived McKay correspondence} \cite{Kleinian_derived} for $Y$ and the theory of cohomological (K-theoretical, categorified) Hall algebras in this setting. The former is an equivalence between the bounded derived category of coherent sheaves on $Y$ and the bounded derived category of representations of the \textit{preprojective algebra} $\Pi_{\calQ}$  of $\calQ$.\footnote{In the present paper, we will use the point of view developed in \cite{VdB_Flops}.} One can define  (convolution) algebra structures on the homology (resp.\ K-theory) of the moduli stack of coherent sheaves on $Y$ and of the moduli stack of finite-dimensional representations of the preprojective algebra, respectively: these are two examples of \textit{two-dimensional cohomological} (resp.\ \textit{K-theoretical}) \textit{Hall algebras}. It is natural to wonder what kind of relation the derived McKay correspondence establishes at the level of these algebras. 

In the present paper, we will prove that the derived McKay correspondence induces isomorphisms between these algebras, after restricting to the \textit{semistable ones}. In addition, we will describe the relation between the derived McKay correspondence and the categorification of these algebras (called \textit{categorified Hall algebras}), recently introduced in \cite{Porta_Sala_Hall, DPS_Torsion_Pairs}. We will view all these relations as ``versions'' of the (derived) McKay correspondence for (semistable) cohomological, K-theoretical, and categorified Hall algebras.

\subsection{Hall algebras}

Let $\bfCoh_{\mathsf{ps}}(Y)$ be the derived moduli stack of properly supported coherent sheaves on $Y$. Given $\mu \in \Q_{>0}$ and a $\Q$-polarization\footnote{This means that there exists a positive integer $r$ such that $r\omega$ is integral and ample.} $\omega$ on $Y$, we denote by $\bfCoh_{\mu}^{\omega\textrm{-}\mathsf{ss}}(Y)$  the derived moduli stack of $\omega$-semistable properly supported coherent sheaves of dimension one on $Y$ of slope $\mu$. By abuse of notation, we will denote by $\bfCoh_{\infty}^{\omega\textrm{-}\mathsf{ss}}(Y)$ the derived moduli stack of zero-dimensional coherent sheaves on $Y$. 

Let $\scrX$ be one of the derived stacks above. As shown in \cite{Porta_Sala_Hall}, the dg-category\footnote{Here, by dg-category we mean \textit{stable $\infty$-category}. Moreover, $\catCohb(-)$ denotes the dg-category of locally cohomologically bounded complexes with coherent cohomology.} $\catCohb( \scrX ) $ has the structure of an $\mathbb E_1$-monoidal dg-category (a \textit{categorified Hall algebra} of $\scrX$) which induces, after passing to K-theory, a structure of an associative algebra on $\sfG_0( \scrX )$ (a \textit{K-theoretical Hall algebra} of $\scrX$). By using the construction of Borel-Moore homology for higher stacks developed in \cite{COHA_surface}, one can show that $\sfH_\ast^{\mathsf{BM}}( \scrX )$ has the structure of an associative algebra (a \textit{cohomological Hall algebra} of $\scrX$). 
%All these results hold also equivariantly with respect to the natural $\C^\ast\times\C^\ast$-action on these stacks induced by the torus action on $Y$. 
In the K-theoretical and cohomological cases, these algebras coincide with those constructed by Kapranov and Vasserot in \cite{COHA_surface} and, for $G=\Z_2$, with those defined in \cite{Sala_Schiffmann}.  

As a necessary step in providing a McKay correspondences for the $\mathbb E_1$-monoidal dg-category and the algebras descrived above, we need a ``quiver'' counterpart of the constructions in \cite{Porta_Sala_Hall}. This has been done in \cite{DPS_Torsion_Pairs} in greater generality. Let $\bfCoh_{\mathsf{ps}}(\Pi_\calQ)$ be the derived moduli stack of Toën-Vaquié's pseudo-perfect objects of the dg-category $\Pi_\calQ\Mod$ of right $\Pi_\calQ$-modules (i.e., representations of $\Pi_\calQ$), which are flat with respect to the standard $t$-structure: this is a derived enhancement of the classical moduli stack $\mathcal{R}ep(\Pi_\calQ)$ of finite-dimensional representations of $\Pi_\calQ$. Given a stability condition $\zeta \in \Q^{\calQ_0}$ and a slope $\vartheta\in \Q$, let $\bfCoh_\vartheta^{\zeta\textrm{-}\mathsf{ss}}(\Pi_\calQ)$ denote the open derived substack of $\zeta$-semistable representations of fixed $\zeta$-slope $\vartheta$. Thanks to the results in \textit{loc.cit.}, there exist categorified, K-theoretical, and cohomological Hall algebras associated with both $\bfCoh_{\mathsf{ps}}(\Pi_\calQ)$ and $\bfCoh_\vartheta^{\zeta\textrm{-}\mathsf{ss}}(\Pi_\calQ)$.
The cohomological and K-theoretical Hall algebras of $\Pi_\calQ$ can alternatively be obtained via the explicit description of the truncation $\mathcal{R}ep(\Pi_\calQ)$ of $\bfCoh_{\mathsf{ps}}(\Pi_\calQ)$ via moment maps instead of using derived algebraic geometry, see \cite{SV_generators}.
As shown by Davison in \cite[Appendix]{RS_Hall} (see also \cite{YZ_2Hall}), a dimensional reduction argument allows to realize the cohomological Hall algebra of $\Pi_\calQ$ as a \textit{three-dimensional} Kontsevich-Soibelman cohomological Hall algebra \cite{KS_Hall} of the Jacobi algebra of a quiver $\widetilde{\calQ}$ with a potential $W$, canonically associated to $\calQ$. 

\subsection{McKay correspondence}

By \cite{VdB_Flops,NN, N_derived}, the bounded derived category of coherent sheaves $\catDb(\catCoh(Y))$ has a local projective generator $\calP$. This determines by tilting an equivalence of derived categories 
\begin{align}\label{eq:tau-equivalence}
	\tau \colon \catDb(\catCoh(Y)) \xrightarrow{\sim} \catDb(\catMod(\Pi_\calQ))\ , 
\end{align}
which we see as a \textit{derived McKay correspondence}\footnote{An alternative approach to derived McKay correspondence via Fourier-Mukai functors was developed in \cite{Kleinian_derived,BKR}.}, where $\catMod(\Pi_\calQ)$ denotes the abelian category of representations of $\Pi_\calQ$. The equivalence $\tau$ admits a natural enhancement at the level of dg-categories.

In the following, we denote by $C_i$ the $i$-th irreducible component of $\pi^{-1}(0)$ for $1\leq i \leq r$, and by $\langle -, - \rangle$ the intersection pairing on the Picard group $\Pic(Y)$, which extends canonically to $\Pic_\Q(Y)\coloneqq \Pic(Y)\otimes_\Z\Q$.  

By analyzing the behavior\footnote{The study of moduli spaces (and stacks) of semistable coherent sheaves on smooth (quasi-)projective complex surfaces via representations of quivers has a long history that we can trace back to the work of Drezet and Le Potier \cite{DLP_StableSheaves}.} of semistability under $\tau$ and applying the formalism developed in \cite{Porta_Sala_Hall, DPS_Torsion_Pairs} to this setting, we prove our main result: 
\begin{theoremintroduction}\label{thm:mainone}
	Let $\omega$ be a $\Q$-polarization of $Y$ and $\mu\in \Q_{>0}\cup\{\infty\}$. Set
	\begin{align}
			\zeta_i\coloneqq \langle \omega, C_i\rangle
	\end{align}
	for $1\leq i \leq r$, and
	\begin{align}
		\zeta_{r+1}\coloneqq\begin{cases}
	\displaystyle	\frac{1}{\mu}  - \sum_{i=1}^{r}\, \zeta_i\ & \text{if } \mu\neq \infty\ ,\\[8pt]
	\displaystyle	- \sum_{i=1}^{r}\, \zeta_i &\text{otherwise}\ .
		\end{cases}
	\end{align}
		Then the tilting functor induces a monoidal equivalence
		\begin{align}
			\catCohb( \bfCoh_{\mu}^{\omega\textrm{-}\mathsf{ss}}(Y) )\simeq \catCohb( \bfCoh^{\zeta\textrm{-}\mathsf{ss}}_0(\Pi_\calQ) )
		\end{align}
		as $\mathbb E_1$-monoidal dg-categories. Therefore, it induces isomorphisms of associative algebras\footnote{Here and it was follows, we prefer to use derived stacks, although the $G_0$-theory and Borel-Moore homology of a derived stack coincide with those of its classical truncation.}
		\begin{align}
			\sfG_0( \bfCoh_{\mu}^{\omega\textrm{-}\mathsf{ss}}(Y)  ) \simeq \sfG_0( \bfCoh^{\zeta\textrm{-}\mathsf{ss}}_0(\Pi_\calQ) ) \quad\text{and}\quad \sfH_\ast^{\mathsf{BM}}( \bfCoh_{\mu}^{\omega\textrm{-}\mathsf{ss}}(Y)  ) \simeq \sfH_\ast^{\mathsf{BM}}( \bfCoh^{\zeta\textrm{-}\mathsf{ss}}_0(\Pi_\calQ) )\ .
		\end{align}
		Moreover, the same holds equivariantly  with respect to a diagonal torus $T\subset \GL(2,\C)$ centralizing the finite group $G$.
\end{theoremintroduction}

\subsection{Type A case}

When $G=\Z_{N+1}$, with $N\geq 1$, we can provide a refined version of our result above. Fix $\mu\in \Q_{>0}$. Then the set $\sfS_\mu$ of isomorphism classes of $\omega$-stable properly supported sheaves of dimension one with slope $\mu$ is finite. Moreover, by Lemma~\ref{finitewallA}, each such sheaf $\calF$ is scheme-theoretically supported on a connected divisor $C_{i,j}\coloneqq\sum_{\ell=i}^j C_\ell$ with $1\leq i \leq j \leq N$. The divisor associated to an isomorphism class $\alpha \in \sfS_\mu$ will be denoted by $C_{i_\alpha, j_\alpha}$. An ordered sequence of isomorphism classes $(\alpha_1, \ldots, \alpha_s)$ is said to be a \textit{chain} if and only if 
\begin{itemize}[leftmargin=0.6cm]\itemsep0.2cm
	\item[(a)] $j_{\alpha_t} = i_{\alpha_{t+1}}$ for all $1\leq t \leq s-1$, and 
	\item[(b)] the given sequence is maximal with this property. 
\end{itemize}
Then the set $\sfS_\mu$ admits a unique partition into chains $\balpha_c$, with $1\leq c \leq C_\mu$. Let $s_c$ denote the length of the chain $\balpha_c$. 

Now, Corollaries~\ref{chaincoroA} and \ref{chaincorB} ensure us that there is a block decomposition of the category of semistable coherent sheaves on Y (and likewise for the preprojective algebra side). Combining a categorical result for chains of rational curves proven in \cite{Tilting_chains} (cf.\ \S\ref {chainsect}) with the formalism in \cite{Porta_Sala_Hall, DPS_Torsion_Pairs}, we obtain the following:
\begin{theoremintroduction} \label{thm:maintwo} 
	Let $\omega$ be a polarization of $Y$ and $\mu\in \Q_{>0}$.	Then we have a monoidal functor
	\begin{align}
		\bigotimes_{c=1}^{C_\mu}\, \catCohb( \bfCoh(\Pi_{A_{s_c}}) ) \longrightarrow	\catCohb( \bfCoh_\mu^{\omega\textrm{-}\mathsf{ss}}(Y) ) 
	\end{align}
	of $\mathbb E_1$-monoidal dg-categories. It induces isomorphisms of associative algebras
	\begin{align}\label{eq:isomorphisms}
		\sfG_0( \bfCoh_\mu^{\omega\textrm{-}\mathsf{ss}}(Y)  ) \simeq \bigotimes_{c=1}^{C_\mu}\,  \sfG_0( \bfCoh(\Pi_{A_{s_c}}) ) \ ,\ \sfH_\ast^{\mathsf{BM}}( \bfCoh_\mu^{\omega\textrm{-}\mathsf{ss}}(Y)  ) \simeq \bigotimes_{c=1}^{C_\mu}\, \sfH_\ast^{\mathsf{BM}}( \bfCoh(\Pi_{A_{s_c}}) )\ .
	\end{align}
	Moreover, these results holds also equivariantly with respect to a diagonal torus $T\subset \GL(2,\C)$ centralizing the finite group $G$.
\end{theoremintroduction}
In the equivariant setting, we can give a representation-theoretic interpretation of the isomorphisms \eqref{eq:isomorphisms} in terms of finite Yangians and quantum affine algebras of type ADE, using \cite{SV_Yangians, VV_Quantumloop}, respectively.

Let $\omega$ be a polarization of $Y$ and $\mu\in \Q_{>0}$.	Set
\begin{align}
	\zeta_i\coloneqq \langle \omega, C_i\rangle \text{ for } 1\leq i \leq N\ , \quad\text{ and }\quad \zeta_{N+1}\coloneqq 	\displaystyle	\frac{1}{\mu}  - \sum_{i=1}^{N}\, \zeta_i\ .
\end{align}
Then, by using Theorem~\ref{thm:mainone}, one has a version of Theorem~\ref{thm:maintwo} in the quiver setting, by replacing $ \bfCoh_\mu^{\omega\textrm{-}\mathsf{ss}}(Y)$ with $\bfCoh^{\zeta\textrm{-}\mathsf{ss}}_0(\Pi_\calQ) $.

\subsubsection{Betti numbers and restricted Kac polynomials}

By using Theorem \ref{thm:maintwo} and the characterization of the generating function of Betti numbers of moduli stacks of semistable representations of the preprojective algebra via a plethystic exponential involving the corresponding \textit{Kac's polynomial} given e.g. in \cite[\S1.8.2 and \S.19]{Davison_integrality}, we obtain an explicit formula for the Kac's polynomial $a^{\zeta\textrm{-}\mathsf{ss}}_{\calQ, \bfd}(t)$ corresponding to the stack $\bfCoh_0^{\zeta\textrm{-}\mathsf{ss}}(\Pi_\calQ; \bfd)$, as we shall explain now.

For any $1\leq c \leq C_\mu$, let $\Delta_c^{+}$ denote the set of positive roots of the associated Lie algebra of type $A_{s_c}$. For $\lambda_c = \sum_{\ell=i}^j\, \alpha_{c, \ell}$, with $1\leq c\leq C_\mu$, set
\begin{align}
	\bfd(\lambda_c)  \coloneqq \sum_{\ell=i}^j\, \bfd(\alpha_{c, \ell})\ ,
\end{align}
where $\alpha_{c,1}, \ldots, \alpha_{c, s_c}$ are the simple roots of $A_{s_c}$ and $\bfd(\alpha_{c, \ell})\in \Z^{N+1}$ denotes the dimension vector of $\alpha_{c, \ell}$, seen as an element of $\sfS_\mu$, for $\ell=1, \ldots, s_c$. As shown in \S\ref{ss:betti_numbers}, by using Theorem~\ref{thm:maintwo} together with Davison's results \cite{Davison_integrality}, one gets:
\begin{align}
	\sum_{\bfd\in \Z^{N+1}, \, \bfd\neq 0}\, a^{\zeta\textrm{-}\mathsf{ss}}_{\calQ,\, \bfd}(q^{-1/2}) y^\bfd
	= \sum_{c=1}^{C_\mu}\, \sum_{\lambda_c\in \Delta^+_c}\, y^{\bfd(\lambda_c)}\ ,
\end{align} 
where $y=(y_1, \ldots, y_{N+1})$ are formal variables, and $y^\bfd \coloneqq \prod_{i=1}^{N+1} y_i^{d_i}$ for any $\bfd \in \Z^{N+1}$. 

\subsection{Affine Yangians, t-structures and Hall algebras}

We conclude this introduction by discussing the ``nature'' of the categorified (K-theoretical, cohomological) Hall algebra of $\bfCoh(Y)$.

In the theory of classical Hall algebras, an equivalence $\tau\colon \catDb(\calA)\to \catDb(\calB)$ between the bounded derived categories of two abelian categories $\calA$ and $\calB$ does not necessarily lift to an isomorphism between the corresponding Hall algebras. For example, one knows that the derived category of representations of the Kronecker quiver is equivalent to the bounded derived category of coherent sheaves on $\PP^1$. The corresponding Hall algebras are not isomorphic but they realize different halves of the quantum loop algebra of $\fraksl(2)$: the former is the Drinfeld-Jimbo positive half (cf.\ \cite{Ringel_Hall, Green_Hall}), while the latter is the positive half with respect to the so-called Drinfeld' new presentation, as proved in \cite{Kapranov_Hall}. Cramer's theorem \cite{Cramer_Hall} shows that the derived equivalence yields to an isomorphism between the corresponding (reduced) Drinfeld doubles. In the previous example, Cramer's theorem provides an algebro-geometric realization of Beck's isomorphism \cite{Beck_Quantum} between the two different realizations of the quantum loop algebra of $\fraksl(2)$.

By analogy with what we recalled above, we do not expect that the McKay equivalence \eqref{eq:tau-equivalence} induces an isomorphism between the categorified (K-theoretical, cohomological) Hall algebras associated to the moduli stacks $\bfCoh_\mathsf{ps}(Y)$ and $\bfCoh_\mathsf{ps}(\Pi_\calQ)$, respectively. This is because the two algebras are associated to two different t-structures on the dg-category $\catQCoh(Y)$, the standard one and the one induced by $\tau$ - called \textit{perverse t-structure}. Both algebras are induced by deeper structures at the level of the corresponding moduli stacks $\bfCoh_\mathsf{ps}(Y)$ and on $\bfCoh_\mathsf{ps}(\Pi_\calQ)$: these structures go under the name of \textit{Dyckerhoff-Kapranov's 2-Segal space} \cite{Dyckerhoff_Kapranov_Higher_Segal}. Both stacks $\bfCoh_\mathsf{ps}(Y)$ and $\bfCoh_\mathsf{ps}(\Pi_\calQ)$ are open substacks of  Toën-Vaquié's moduli of objects $\scrM_{\catPerf(Y)}$ and the natural 2-Segal space on the latter induces the 2-Segal space structures on the former. 

Since the cohomological Hall algebra of $\bfCoh_\mathsf{ps}(\Pi_\calQ)$ realizes the positive half, in the Drinfeld's presentation of the Maulik-Okounkov Yangian of affine type ADE, from the perspective described above, it is natural to expect that the cohomological Hall algebra of $\bfCoh_\mathsf{ps}(Y)$ should realize a \textit{completely new} positive half of (a version of) the Yangian of affine type ADE.

The discussion above justifies the study of the cohomological Hall algebra of $\bfCoh_\mathsf{ps}(Y)$ from an algebraic point of view.
From a more geometric perspective, the study this cohomological Hall algebra can be understood within the context of a vaster program, whose aim is to provide a cohomological version of the \textit{algebra of Hecke operators} arising from Ginzburg-Kapranov-Vasserot's the geometric Langlands correspondence for surfaces \cite{GKV_Langlands}. 

The present paper represents a first step in the  direction of understanding the whole cohomological Hall algebra of $\bfCoh_\mathsf{ps}(Y)$, which will be pursued further in the future.

\subsection{Outline}

\S\ref{setupsection} is devoted to the theory of perverse coherent sheaves on a resolution $Y$ of a Kleinian singularity and the study of the relation between their semistability and the semistability of coherent sheaves on $Y$. In \S\ref{s:Cat-McKay} we provide a version of the McKay correspondence for semistable categorified, K-theoretical, and cohomological Hall algebras. There are also two appendices: one about a realization of the perverse $t$-structure via tilting, following \cite{YoshiokaPerverse} and another one about some characterizations of sheaves on surfaces.

\subsection*{Acknowledgments}

We would like to thank Ben Davison, Greg Moore, Tony Pantev, Sarah Scherotzke, Olivier Schiffmann, Yan Soibelman, and Éric Vasserot for enlightening conversations. We also thank Martin Kalck to pointing out an error in the first version of the paper and providing comments. Finally, we thank Éric Vasserot for his encouragement to generalize the first version of the paper dealing with type A quivers to the current one dealing with type ADE quivers. 

Part of this work was initiated while the third-named author was visiting Rutgers University. He is grateful to this institution for their hospitality and wonderful working conditions.

\section{Perverse coherent sheaves on resolutions of quotient singularities}\label{setupsection}

In this section, we will recall basic geometric properties of the resolution $Y$ of the type ADE singular surface and the (categorical) McKay correspondence, establishing a relation between the bounded derived category of coherent sheaves on $Y$ and the bounded derived category of representations of the preprojective algebra of the affine type ADE quiver.

\subsection{Geometry of the resolutions}\label{ss:resolution}

First, recall that all finite subgroups $G$ of $\mathsf{SL}(2, \C)$ are classified by Dynkin diagrams of finite type ADE.

Fix a finite subgroup $G$ of $\mathsf{SL}(2, \C)$ and let $\calQ^{\mathsf{fin}}=(\calQ^{\mathsf{fin}}_0, \calQ^{\mathsf{fin}}_1)$ be the corresponding Dynkin diagram. We denote by $\frakg$ the simple Lie algebra associated to $\calQ^{\mathsf{fin}}$. Let $\alpha_1, \ldots, \alpha_r$ be the simple roots of $\frakg$, where $r$ is the dimension of the Cartan subalgebra of $\frakg$, and let
\begin{align}
	\alpha_{r+1}=\sum_{i=1}^r\, m_i \alpha_i
\end{align}
be the highest root of $\frakg$. 
\begin{remark}
%	For completeness, we list them:
%	\begin{align}
%		A_N& : \quad \dynkin[labels={1, 2, 3, 4, 5, , , N-1, N}, label macro/.code={\alpha_{\drlap#1}}, edge length=1.2cm]A9\quad;\qquad N\geq 1\\[4pt]
%		D_N& : \quad \dynkin[labels={1, 2, 3, 4, , , N-2, N-1, N}, label macro/.code={\alpha_{#1}}, edge length=1.2cm]D9\quad;\qquad  N\geq 4\\[4pt]
%		E_6 & : \quad \dynkin[labels={1, 6, 2, 3, 4, 5}, label macro/.code={\alpha_{\drlap#1}}, edge length=1.2cm]E6\\[4pt]
%		E_7 & : \quad \dynkin[labels={1, 7, 2, 3, 4, 5, 6}, label macro/.code={\alpha_{\drlap#1}}, edge length=1.2cm]E7\\[4pt]
%		E_8 & : \quad \dynkin[labels={1, 8, 2, 3, 4, 5, 6, 7}, label macro/.code={\alpha_{\drlap#1}}, edge length=1.2cm]E8
%	\end{align}
	The explicit values of $r$ and the $m_i$'s are:
	\begin{align}
		\begin{array}{|c|c|c|}
			\hline
			\frakg & r & m_1, m_2, m_3, \ldots, m_{r-1}, m_r\\		
			\hline
			A_N & N & 1, \ldots, 1\\
			\hline
			D_N & N & 1, 2, 2, \ldots, 2, 1, 1\\
			\hline
			E_6 & 6 & 1, 2, 3, 2, 1, 2\\
			\hline
			E_7 & 7 & 1, 2, 4, 3, 2, 1, 2\\
			\hline
			E_8 & 8 & 2, 4, 6, 5, 4, 3, 2, 3\\
			\hline
		\end{array}
	\end{align}
\end{remark}

The quotient $X\coloneqq\C^2/G$ has an isolated singularity at the origin. Let $\pi\colon Y\to X$ the minimal resolution of singularities of $X$. Let $C\coloneqq \pi^{-1}(0)$ and we denote by $C_{\mathsf{red}}$ its reduced variety. The irreducible components $C_i$'s of $C$ are isomorphic to $\PP^1$.
\begin{proposition}\label{prop:McKay}
	Let $\calQ=(\calQ_0, \calQ_1)$ be the affine extension of $\calQ^{\mathsf{fin}}$.
	\begin{enumerate}\itemsep0.2cm
		\item \label{Mckay-(finite)} The vertices of $\calQ^{\mathsf{fin}}$ are in canonical bijection with the irreducible components $C_i$ of $C$. Two vertices are joined by an edge if and only if the corresponding components intersect. The intersection is transverse and consists of one point. Moreover, the intersection matrix of the $C_i$'s is equal to the opposite of the Cartan matrix of $\calQ^{\mathsf{fin}}$.
		\item The vertices of $\calQ$ are in bijection with the irreducible representations of $G$.
	\end{enumerate}
\end{proposition}

\begin{remark}
	Note that we have the following equality in $\Pic(Y)$:
	\begin{align}
		C= \sum_{i=1}^r\, m_i C_i\ .
	\end{align}
\end{remark}

The following basic fact will be used repeatedly throughout this paper. 
\begin{lemma}\label{compactdiv} 
	Let $Z\subset Y$ be a proper divisor. Then $Z_{\mathsf{red}}\subseteq C_{\mathsf{red}}$. Moreover, $H_2(Y; \Z)$ is generated over $\Z$ by the irreducible components $C_1, \ldots, C_r$. 
\end{lemma}  
%
%\begin{proof} The first statement follows from the fact that $X$ is affine. 
%\end{proof}

We fix a nontrivial diagonal torus $T\subseteq \C^\ast\times\C^\ast \subset \mathsf{GL}(2,\C)$ centralizing $G$. Depending on $G$, we have $T=\C^\ast$ or $T=\C^\ast \times \C^\ast$: for example, in type $A$ both cases are admissible, while in type $D$ only the former. The group $T$ acts on $X$ in the obvious way and the $T$-action on $X$ lifts to an $T$-action on $Y$ such that the map $\pi$ is $T$-equivariant\footnote{One may see this using e.g. the Nakajima quiver varieties interpretation of the map $\pi$.}.

\subsection{Perverse coherent sheaves on the resolution}\label{ss:psheavesres} 

Let $\scrC$ be the so-called \textit{null category}, i.e., the abelian subcategory of $\catCoh(Y)$ consisting of coherent sheaves $\calF$ such that $\R\pi_\ast \calF=0$. Following \cite[\S3.1]{VdB_Flops}, we introduce the following \textit{torsion pair}\footnote{In \textit{loc.cit.} it is denoted by $(\scrT_{-1}, \scrF_{-1})$.} on $\catCoh(Y)$:
\begin{align}
	\scrT&\coloneqq \{\calF\in \catCoh(Y)\, \vert\, \R^1\pi_\ast \calF=0\,\text{ and }\, \Hom(\calF, \scrC)=0\}\ ,\\[2pt]
	\scrF&\coloneqq \{\calF\in \catCoh(Y)\, \vert\, \R^0\pi_\ast \calF=0\}\ .
\end{align}
We denote by $\catP(Y/X) \subset \catDb(\catCoh(Y))$  the heart of the tilted $t$-structure induced by $(\scrT, \scrF)$, which is called the \textit{Bridgeland's perverse $t$-structure}.
\begin{definition}\label{def:perverse}
	We say that an object $E$ of $\catDb(\catCoh(Y))$ is a \textit{perverse coherent sheaf} on $Y$ if $E\in \catP(Y/X)$, i.e., $E$ satisfies the following conditions:
	\begin{enumerate}\itemsep0.2cm
		
		\item \label{item:perverse-(1)} all cohomology sheaves $\calH^i(E)$ are zero except for $i=-1, 0$,
		
		\item \label{item:perverse-(2)} $\calH^{-1}(E)\in \scrF$, and
		
		\item \label{item:perverse-(3)} $\calH^0(E)\in \scrT$. 
		
	\end{enumerate} 
\end{definition}

Consider now the abelian subcategory $\catCohps(Y)$ of $\catCoh(Y)$ consisting of properly supported coherent sheaves. Then the pair $(\scrT\cap \catCohps(Y), \scrF\cap \catCohps(Y))$ is a torsion pair of $\catCohps(Y)$. We denote by $\catPps(Y/X)$ the corresponding tilted heart in $\catDb(\catCohps(Y))\simeq \catDbps(\catCoh(Y))$, where the latter is the subcategory of $\catDb(\catCoh(Y))$ consisting of complexes with properly supported cohomology. $\catPps(Y/X)$ is the abelian subcategory of $\catP(Y/X)$ consisting of objects with proper support. 

\begin{remark}\label{rem:yoshioka}
	As described in \cite[\S2.1]{YoshiokaPerverse} (and revisited in Appendix~\ref{s:yoshioka}), one has the following characterization:
	\begin{align}
		\scrT\cap \catCohps(Y) = &\{\calF\in \catCohps(Y)\, \vert\, \Hom(\calF, \scrO_{C_i}(-1))=0 \text{ for } i=1, \ldots, e\} \ ,\\[2pt]
		\scrF\cap \catCohps(Y)  = &\{\calF\in \catCohps(Y)\, \vert\, \calF \text{ is a successive extension of subsheaves of the }\\[2pt] 
		&\scrO_{C_i}(-1)\text{'s for }i=1, \ldots, e\} \ . 	
	\end{align}
	As a consequence, we see that $\calH^{-1}(E)$ is a pure one-dimensional coherent sheaf for $E\in \catPps(Y/X)$.
\end{remark}

Following \cite[Lemma~3.4.4]{VdB_Flops}, let $D_i$ be the unique Cartier divisor of $Y$ such that $D_i \cdot C_j=\delta_{i, j}$. Let
\begin{align}\label{eq:tiltlb}
	L_i \coloneqq \scrO_Y(D_i)
\end{align}
be the corresponding line bundle, for $i=1, \ldots, r$. Let $\calE_i$ be the locally free sheaf obtained as the extension
\begin{align}\label{eq:projgenZN_components}
	0\longrightarrow \scrO_Y^{m_i-1}\longrightarrow \calE_i\longrightarrow L_i \longrightarrow 0
\end{align}
associated to a minimal set of $m_i - 1$ generators of $H^1(Y, L_i^{-1})$. 

\begin{notation}
	Define $\calE_{r+1} \coloneqq \scrO_Y$ and set $m_{r+1}\coloneqq 1$.
\end{notation}

\begin{remark}
	By construction, the $\calE_i$'s are the locally free sheaves defined by Gonzalez-Sprinberg and Verdier \cite{GSV_McKay}. As shown in \textit{loc.cit.}, they generate, over $\Z$, the Grothendieck group of locally free sheaves on $Y$. 
\end{remark}
By \cite[Theorem~3.5.5]{VdB_Flops}, each $\calE_i$ is an indecomposable projective object in $\catP(Y/X)$ and the direct sum 
\begin{align}\label{eq:projgenZN}
	\calP \coloneqq \bigoplus_{i=1}^{r+1}\, \calE_i
\end{align}
is a projective object and a generator of $\catP(Y/X)$. It is well known that $A\coloneqq \End( \calP )$ is isomorphic to the preprojective algebra $\Pi_{\calQ}$ of the affine Dynkin diagram $\calQ$ corresponding to $G$ (see e.g. \cite{CH_DeformedPreprojectiveAlgebras, Wemyss_McKay}). By combining this with \cite[Corollary~3.2.8]{VdB_Flops}, the functors 
\begin{align}\label{eq:dercateqB} 
	\tau\coloneqq \R\Hom(\calP,-)\colon \catDb(\catCoh(Y)) \longrightarrow \catDb(\catMod(\Pi_{\calQ}))\\[4pt]
	\tau\coloneqq \R\Hom(\calP,-)\colon \catDbps(\catCoh(Y)) \longrightarrow \catDbps(\catMod(\Pi_{\calQ})) 
\end{align}	
and
\begin{align}
	(-) \otimes^{\LL}_{\Pi_{\calQ}} \calP \colon \catDb(\catMod(\Pi_{\calQ})) &\longrightarrow \catDb(\catCoh(Y)) \\[4pt]
	\quad (-) \otimes^{\LL}_{\Pi_{\calQ}} \calP \colon \catDbps(\catMod(\Pi_{\calQ})) &\longrightarrow \catDbps(\catCoh(Y)) 
\end{align}
determine an equivalence of derived categories. Here, $\catMod(\Pi_{\calQ})$ denotes the abelian category of finitely generated right $\Pi_\calQ$-modules\footnote{As usual, we call a \textit{representation of $\Pi_\calQ$} a right $\Pi_\calQ$-module.}. This restricts to an equivalence of hearts
\begin{align}\label{eq:cateqA} 
	\catP(Y/X) \xrightarrow{\sim}  \catMod(\Pi_{\calQ})\ ,
\end{align}
as well as  
\begin{align}\label{eq:cateqB}
	\catPps(Y/X) \xrightarrow{\sim}  \mathsf{mod}(\Pi_{\calQ})\ ,
\end{align}
where $\mathsf{mod}(\Pi_\calQ)$ is the subcategory consisting of $\Pi_{\calQ}$-modules which are finite dimensional as complex vector spaces. 
\begin{remark}
	Assume that the torus $T$ acts on a representation $M$ of $\Pi_\calQ$ in the following way:
	\begin{align}\label{action1}
		(q_1,q_2)\cdot x_h\coloneqq q_hx_h		
	\end{align}
	where $x_h$ denotes the linear map in $M$ corresponding to an edge $h$ belonging to the double quiver $\overline{\calQ}$, $q_h\coloneqq q_1$ if $h\in \calQ_1$ and $q_h\coloneqq q_2$, otherwise. This action lifts to an action at the level of $\catMod(\Pi_{\calQ})$ and its bounded derived category. Therefore, $\tau$ and the equivalences \eqref{eq:cateqA} and \eqref{eq:cateqB} are equivariant with respect to the action of $T$.
\end{remark}

\begin{remark}
	Note that the equivalence \eqref{eq:dercateqB} induces an isomorphism at the level of Grothendieck groups: $\tau\colon \sfK_0(\catCoh(Y))\simeq \sfK_0(\catMod(\Pi_{\calQ}))$ and $\tau\colon \sfK_0(\catCohps(Y))\simeq \sfK_0(\mathsf{mod}(\Pi_{\calQ}))$. In particular, the Euler forms are identified. Recall that the Euler form on $\sfK_0(\mathsf{mod}(\Pi_{\calQ}))\simeq \Z^{\calQ_0}$ has the form:
	\begin{align}\label{eq:pairing}
		\langle\bfd, \bfe\rangle\coloneqq \sum_{\imath\in\calQ_0}\, d_\imath e_\imath -\sum_{a\in \calQ_1}\, d_{\sfs(a)}e_{\sft(a)}\ ,
	\end{align}
	while the symmetrized Euler form is $(\bfd, \bfe)\coloneqq \langle \bfd, \bfe\rangle + \langle \bfe, \bfd\rangle$.
\end{remark}

\begin{remark}\label{rem:simple}
	By \cite[Proposition~3.5.7]{VdB_Flops} the simple modules $\calS_1, \ldots, \calS_{r+1}$ associated to the nodes of the affine quiver correspond to the \textit{spherical objects} 
\begin{align}\label{eq:sphericalobj}
	\calS_{r+1} \simeq \tau(\scrO_{C})\quad\text{and}\quad	\calS_i \simeq \tau(\scrO_{C_i}(-1)[1]) \ ,
\end{align}
for $1\leq i \leq r$. Set
\begin{align}
	\calI_j\coloneqq \begin{cases}
		\scrO_{C} & \text{for } j=r+1\ ,\\
		\scrO_{C_j}(-1)[1] & \text{for } j=1, \ldots, r\ .
	\end{cases}
\end{align}
Then, they satisfy the orthogonality relations 
\begin{align}\label{eq:orthrel} 
	\R\Hom(\calE_i, \calI_j) = \delta_{i,j}\, {\underline \C}\ , %_{\catDb(\catCoh(Y))}
\end{align}
for $1\leq i, j \leq r+1$.
\end{remark}

Given an object $E\in \catPps(Y/X)$, the categorical equivalence \eqref{eq:dercateqB} implies that  
$\R\Hom(\calE_k, E)$ is a one-term complex of amplitude $[0, \, 0]$, which we denote by $V_k$, for any $k=1, \ldots, r+1$. Thus, an object $E$ of $\catPps(Y/X)$ is mapped to a representation $M$ of $\Pi_{\calQ}$ with underlying $\Z_{r+1}$-graded vector space 
\begin{align}\label{eq:vectorspace}
	V\coloneqq \bigoplus_{k=1}^{r+1}\ V_k\ , 
\end{align}
where the $k$-th summand has degree $k$. We denote by $d_k(E)$ the dimension of $V_k$ for $1\leq k \leq r+1$. Moreover, since $E$ is properly supported, one has
\begin{align}\label{eq:ch1}
	\sfch_1(E)\coloneqq -\sfch_1(\calH^{-1}(E))+ \sfch_1(\calH^{0}(E))=\sum_{i=1}^{r}\, n_i C_i\quad\text{and}\quad \chi(\scrO_Y, E)=n\in \Z\ ,
\end{align}
for some $n_i\in \Z$, $i=1, \ldots, r$.
\begin{lemma}\label{dimvectlemm} 
	We have 
	\begin{align}\label{eq:dimvect} 
		m_k d_{r+1}(E)-d_k(E)=n_k\ ,
	\end{align}
	for any $1\leq k \leq r$ and 
	\begin{align}\label{eq:dNchi}
		d_{r+1}(E) =n \ .
	\end{align}
\end{lemma}
\begin{proof}
	By \cite[Lemma~3.2.3]{VdB_Flops}, for any object $E\in \catPps(Y/X)$, one has
	\begin{align}
		\Ext^i(\calE_k,E)=0 %_{\catDb(\catCoh(Y))}
	\end{align}
	for all $1\leq k \leq r+1$ and $i>0$, where the objects $\calE_k$ are introduced in \eqref{eq:projgenZN_components} and $\calE_{r+1}\coloneqq \scrO_Y$. Therefore 
	\begin{align}
		m_k d_{r+1}(E)-d_k(E)=m_k\chi(\scrO_Y, E)-\chi(\calE_k, E)\ .
	\end{align}
	Since $E$ is properly supported, one can use the Riemann-Roch theorem (cf.\ \cite[\S1.1, Formula~(1.60)]{BBHR_FM}) to compute 
	\begin{align}
		\chi(\calE_k, E)=m_k\chi(\scrO_Y, E)- \sfch_1(E) \cdot \sfch_1(\calE_k)\ .
	\end{align}
	The assertion follows by using the relations $\sfch_1(\calE_k) \cdot C_i = \delta_{k, i}$ for $1\leq i, k\leq r$.
\end{proof}
\begin{remark}\label{rem:zero-dimensional}
	Let $E\in \catPps(Y/X)$ be a perverse coherent sheaf whose support\footnote{Recall that the support of an object in $\catDb(\catCoh(Y))$ is the union of the supports of its cohomology sheaves.} is zero-dimensional. Thus both $\calH^{-1}(E)$ and $\calH^0(E)$ are zero-dimensional. On the other hand, by the first condition in Definition~\ref {def:perverse}-\eqref{item:perverse-(2)} we get $\calH^{-1}(E)=0$, therefore $E\simeq \calH^0(E)$. Since $\sfch_1(E)=0$, by Equation \eqref{eq:dimvect} the dimension vector of $E$ is of the form $n\cdot (m_1, \ldots, m_r, 1)$, where $n=\chi(\calH^0(E))$. Note that $(m_1, \ldots, m_r, 1)$ is the indivisible imaginary root.
\end{remark}

\subsection{Semistable perverse coherent sheaves and one-dimensional sheaves}\label{s:semistable}

In this section we will establish a relation between one-dimensional sheaves and semistable perverse coherent sheaves on the minimal resolution.

In what follows, we denote by $\catCohps(Y)$ the abelian category of coherent sheaves on $Y$ with proper support. Moreover, we denote by $E$ an object in $\catPps(Y/X)$, while by calligraphic $\calF$ an object in $\catCohps(Y)$, i.e., a coherent sheaf on $Y$ with proper support.

\subsubsection{Stability conditions for perverse coherent sheaves}\label{sheafwalls} 

Fix $\zeta\in \Q^{r+1}$. In what follows, we shall call $\zeta$ a \textit{stability condition}. 
\begin{definition}
	The \textit{($\zeta$-)degree} of a dimension vector $\bfd \in\N^{r+1}\smallsetminus \{0\}$ as
	\begin{align}
		\deg_\zeta(\bfd)\coloneqq \sum_{i=1}^{r+1} \zeta_i d_i \ .
	\end{align}
	The \textit{($\zeta$-)slope} of $\bfd$ is
	\begin{align}
		\mathsf{slope}_\zeta(\bfd)\coloneqq \deg_\zeta(\bfd)/\sum_{i=1}^{r+1} d_i\ .
	\end{align}
	
	Let $M$ be a finite-dimensional representation of $\Pi_\calQ$. Its \textit{degree} (resp.\ \textit{slope}) is the degree (resp.\ slope) of its dimension vector.
\end{definition}

\begin{definition}
	A nonzero finite-dimensional representation $M$ of $\Pi_\calQ$ is \textit{$\zeta$-semistable} if for any subrepresentation $N$ of $M$ we have
	\begin{align}
		\frac{\deg_\zeta(N)}{\sum_{\imath\in \calQ_0}\, \dim N_\imath}\leq \frac{\deg_\zeta(M)}{\sum_{\imath\in \calQ_0}\, \dim M_\imath}\ . 
	\end{align}
	A nonzero representation $M$ is called \textit{$\zeta$-stable} if the strict inequality holds for any nonzero proper subrepresentation $N \subset M$.
\end{definition}

\begin{definition}\label{stabpc} 
We say that $E\in \catPps(Y/X)$ is \textit{$\zeta$-(semi)stable} if and only if the corresponding representation of $\Pi_{\calQ}$ is $\zeta$-(semi)stable. We will also write 
\begin{align}
	\deg_\zeta (E) \coloneqq \sum_{i=1}^{r+1} \zeta_i d_i(E) \ .
\end{align}
\end{definition}

%\begin{remark}\label{stackisomss} 
%	Given Definition~\ref {stabpc} the stack isomorphism in Lemma~\ref {stackisomone} restricts to an isomorphism of moduli stacks of 
%	$\zeta$-semistable objects for any $\zeta\in \IQ^N$. 
%\end{remark} 

The main goal of this section consists of proving some structure results for $\zeta$-semistable perverse sheaves, where $\zeta$ is a suitable stability condition satisfying certain conditions, that we will define in what follows. First note the following consequence of Lemma~\ref {dimvectlemm}. 
\begin{corollary}\label{thetadegcor} 
	Let $\zeta\in \Q^{r+1}$ be a stability condition and let $E$ be an object of $\catPps(Y/X)$. Let $\omega \in \Pic_\Q(Y)\coloneqq \Pic(Y)\otimes_\Z\Q$ be such that	\begin{align}
			\zeta_i= \langle \omega, C_i\rangle 
	\end{align}	
	for $i=1, \ldots, r$. Then one has 
	\begin{align}\label{eq:zetadegA}
		\deg_\zeta(E) = d_{r+1}(E)  \sum_{i=1}^{r+1} m_i\, \zeta_i - \omega \cdot \sfch_1(E)\ .
	\end{align}
	Moreover, 
	\begin{align}\label{eq:dN}
		d_{r+1}(E) =  \dim\, H^0(\calH^0(E)) + \dim\, H^1(\calH^{-1}(E))\ .
	\end{align} 
\end{corollary}

\begin{definition}
	Let $\omega \in \Pic_\Q(Y)\coloneqq \Pic(Y)\otimes_\Z\Q$. We say that $\omega$ is a \textit{polarization}\footnote{This definition is equivalent to the standard one of ampleness because of the definition of the the $D_i$'s.} if 	
	\begin{align}
		\zeta_i= \langle \omega, C_i\rangle >0
	\end{align}	
	for $i=1, \ldots, r$.
\end{definition} 

\begin{rem}\label{rem:sheafchi}
	Let $E\in \catPps(Y/X)$. Then, by combining \eqref{eq:dN} and \eqref{eq:dNchi}, we obtain that $\chi(E)\geq 0$. \hfill $\triangle$
\end{rem}

We introduce the following subcategories of $\catPps(Y/X)$:
\begin{itemize}\itemsep0.2cm
	\item $\catPps^0(Y/X)$ is the full subcategory of $\catPps(Y/X)$ consisting of perverse coherent sheaves $E$ with $\calH^{-1}(E)=0$;
	\item $\catPps^{-1}(Y/X)$ is the full subcategory of $\catPps(Y/X)$ consisting of perverse coherent sheaves $E$ with $\calH^{0}(E)=0$.
\end{itemize}

\begin{definition}
	We say that $E\in \catDb(\catCoh(Y))$ is a \textit{sheaf} if $\calH^i(E)\simeq 0$ for $i\neq 0$.
\end{definition}

\begin{remark}
	Let $E\in \catPps(Y/X)$. Note that $E\in \catPps^0(Y/X)$ if and only $E$ is a sheaf, while $E\in \catPps^{-1}(Y/X)$ if and only if $E[-1]$ is a sheaf. 
%	
%	Moreover, if $E\in \catPps(Y/X)$ is a sheaf, then necessarily $\dim H^1(\calH^0(E))=0$.
\end{remark}

\begin{lemma}\label{exseqlemm}
	\hfill
	\begin{itemize}\itemsep0.2cm
		\item Let $E\in \catPps^0(Y/X)$ and $\calF\coloneqq \calH^0(E)$. Then, any quotient $\calF\twoheadrightarrow \calF''$ in $\catCohps(Y)$ corresponds to an object in $\catPps^0(Y/X)$.
		\item Let $E\in \catPps^{-1}(Y/X)$ and $\calF\coloneqq \calH^{-1}(E)$. Then any subsheaf $\calF'\hookrightarrow \calF$ in $\catCohps(Y)$ corresponds to an object in $\catPps^{-1}(Y/X)$.
	\end{itemize}
\end{lemma} 

\begin{proof}
	Note that $\catPps^{-1}(Y/X)[1]=\scrF\cap \catCohps(Y)$, while $\catPps^0(Y/X)=\scrT\cap \catCohps(Y)$. Hence, the pair $(\catPps^0(Y/X)), \catPps^{-1}(Y/X)[1])$ is a torsion pair of $\catCohps(Y)$. Thus, the result follows from Lemma~\ref{lem:torsion_pairs_extensions_mono_epi}. 
	\end{proof}

Let $\zeta\in \Q^{r+1}$ be a stability condition. Set 
	\begin{align}
	\zeta^\perp \coloneqq \{0\}\cup \big\{\bfd \in \N^{r+1}\setminus \{0\}\, \vert\, \sum_{\imath}\, d_\imath \zeta_\imath = 0\big\} \ .
\end{align}
\begin{proposition}\label{sheafwallA} 
	Let $\zeta\in \Q^{r+1}$ be a stability condition that satisfies the inequality 
	\begin{align}\label{eq:thetaineqA}
		\zeta_{i}>0
	\end{align}
	for $1\leq i\leq r$. Let $E\in \catPps(Y/X)$ be a $\zeta$-semistable object with associated dimension vector $\bfd\in \zeta^\perp$. Then
	\begin{align}
		\begin{cases}
			E\in \catPps^0(Y/X) & \text{if } 	\sum_{i=1}^{r+1}\, m_i\, \zeta_i\geq 0\ , \\[4pt]
			E\in \catPps^{-1}(Y/X) & \text{if } 	\sum_{i=1}^{r+1}\, m_i\, \zeta_i< 0 \ .
		\end{cases}
	\end{align}
	Set $\calF\coloneqq \calH^i(E)$ if $E\in \catPps^i(Y/X)$ for $i=-1, 0$. Then, we have the following two cases:
	\begin{enumerate}\itemsep0.2cm
		\item if $\sum_{i=1}^{r+1}\, m_i\, \zeta_i$ is different from zero, $\calF$ is a purely one-dimensional sheaf;
		\item if $\sum_{i=1}^{r+1}\, m_i\, \zeta_i= 0$, $\calF$ is a zero-dimensional sheaf.
	\end{enumerate} 
\end{proposition}

\begin{proof}
	First note that under the given assumptions, $d_{r+1}(E)\neq 0$ since $\deg_\zeta(E)=0$. Indeed, otherwise we must have $d_i(E)=0$ for $i=1, \ldots, r$ because of condition \eqref{eq:thetaineqA}: thus, $E$ must be the zero object. Therefore, let us assume that $d_{r+1}(E)\neq 0$.
	
	Let us consider the case: 
	\begin{align}\label{eq:thetaineqB}
	\sum_{i=1}^{r+1}\, m_i\, \zeta_i\geq 0\ .
	\end{align}	
	Let $0\neq E' \subsetneq E$ be a subobject in $\catPps(Y/X)$. Since, by assumption, $E$ is $\zeta$-semistable with 
	$\deg_\zeta(E)=0$, one has 
	\begin{align}
		\deg_\zeta(E')\leq 0\ .
	\end{align}
	In particular this holds for $E'=\calH^{-1}(E)[1]$. Then equation \eqref{eq:zetadegA} for $E'$ yields 
	\begin{align}
		\dim\, H^1(\calH^{-1}(E))\,  \sum_{i=1}^{r+1} m_i\, \zeta_i + \omega \cdot \ch_1(\calH^{-1}(E)) \leq 0\ .
	\end{align}
	By construction, $\omega$ is a polarization on $Y$, hence $\omega \cdot \ch_1(\calH^{-1}(E))\geq 0$. Then this yields 
	\begin{align}
		\begin{cases}
			\dim H^1(\calH^{-1}(E)) =0\text{ and }  \ch_1(\calH^{-1}(E))=0 & \text{if }    
			\displaystyle\sum_{i=1}^{r+1} m_i \zeta_i > 0\ ,  \\[12pt]
			\ch_1(\calH^{-1}(E))=0 & \text{if }  \displaystyle\sum_{i=1}^{r+1} m_i \zeta_i = 0\ . 
		\end{cases}
	\end{align}
	In both cases, because of the vanishing of the first Chern class, one has that $\calH^{-1}(E)$ is a zero dimensional sheaf on $Y$. Then the defining properties of $\catPps(Y/X)$ (cf.\ Definition~\ref {def:perverse}-\eqref{item:perverse-(2)}) imply that $\calH^{-1}(E)$ is the zero sheaf (cf.\ Remark~\ref{rem:zero-dimensional}).
	
	Now, suppose that the inequality \eqref{eq:thetaineqB} is strict. Let $\calF'\subset \calF$ be a zero dimensional sheaf in $\catCohps(Y)$. Thanks to Lemma~\ref{exseqlemm}, we can consider the object $E''\coloneqq \calF/\calF'$, together with the morphism $E\to E'' \in \catPps(Y/X)$. Set $E'\coloneqq \mathsf{Cone}(E\to E'')[-1]$. By using equation~\eqref{eq:zetadegA}, we get that 
	\begin{align}
		\deg_\zeta(E') = \dim\, H^0(E')  \sum_{i=1}^{r+1} m_i \zeta_{i}\ .
	\end{align}
	Since $E$ is $\zeta$-semistable of $\zeta$-degree zero, we get $\deg_\zeta(E')\leq 0$, hence $E'=0$ and therefore $\calF'$ is the zero sheaf. Finally, suppose that
	\begin{align}
		\sum_{i=1}^{r+1}\, m_i\, \zeta_i= 0\ .
	\end{align}	
	Then 
	\begin{align}
		0=\deg_\zeta(E) = -\omega \cdot \ch_1(E) \ .
	\end{align}
	Since $E$ is a sheaf and $\omega$ is a polarization on $Y$, this implies $\ch_1(E)=0$. Thus, $\calF$ is zero-dimensional.	
	
	Now, consider the case
	\begin{align}
		\sum_{i=1}^{r+1}\, m_i\, \zeta_i< 0\ .
	\end{align}	
	Let $E\twoheadrightarrow E''$ be a quotient object of $E$ in $\catPps(Y/X)$. Since $E$ is $\zeta$-semistable with 
	$\deg_\zeta(E)=0$, one has 
	\begin{align}
		\deg_\zeta(E'')\geq 0\ .
	\end{align}
	In particular this holds for $E''=\calH^{0}(E)$. Then equation \eqref{eq:zetadegA} for $E''$ yields 
	\begin{align}
		\dim\, H^0(E'')\,  \sum_{i=1}^{r+1} m_i\, \zeta_i - \omega \cdot \ch_1(E'') \geq 0\ .
	\end{align}
Since $\omega$ is a polarization, this yields
\begin{align}
	\dim\, H^0(\calH^{0}(E))=0 \quad\text{and}\quad \ch_1(\calH^{0}(E))=0\ . 
\end{align}
Thus, $\calH^{0}(E)=0$ and therefore $E\in \catPps^{-1}(Y/X)$. We are left to prove that $\calF\coloneqq \calH^{-1}(E)$ is a purely one-dimensional sheaf. Let $\calF'$ be a zero dimensional subsheaf of $\calF$ in $\catCohps(Y)$. Thanks to Lemma~\ref{exseqlemm}, $E'\coloneqq \calF'[1]$ is a subobject of $E$ in $\catPps(Y/X)$. By Definition~\ref {def:perverse}-\eqref{item:perverse-(2)}, we get $\R^0 \pi_\ast \calH^{-1}(E')=0$, thus $\calF'$ must be the zero sheaf.
\end{proof}

Let $\omega$ be a polarization on $Y$.  The \textit{$\omega$-slope} of a coherent sheaf $E\in \catCohps(Y)$ is defined as 
\begin{align}\label{eq:slope-omega}
	\mu_{\omega}(E)\coloneqq \begin{cases}
		\displaystyle \frac{\chi(E)}{\sfch_1(E)\cdot \omega} & \text{if } \sfch_1(E)\cdot \omega\neq 0 \ , \\[4pt]
		 \infty & \text{otherwise}\ .
	\end{cases}\ .
\end{align}
If $\omega$ is of the form
\begin{align}
	\omega = \sum_{i=1}^{r}\,\zeta_i D_i  \ ,
\end{align}
with $\zeta_i>0$ for $i=1, \ldots, r$, Formula~\eqref{eq:slope-omega} reduces to
\begin{align}
	\mu_{\omega}(E) =\frac{\chi(E)}{\sum_{k=1}^{r}\, n_k \zeta_k} \ ,
\end{align}
when $\sfch_1(E)\cdot \omega\neq 0$ and $\sfch_1(E)=\sum_{i=1}^r n_i C_i$. 

\begin{remark}
	Let $E\in\catPps^0(Y/X)$ be a sheaf. Set $\calF\coloneqq \calH^0(E)$. Then
	\begin{align}
		\chi(\calF)= d_{r+1}(E) \quad \text{and} \quad \sfch_1(\calF)=\sfch_1(E)=\sum_{k=1}^r (m_k d_{r+1}(E)-d_k(E))C_k\ .
	\end{align}

	Now, let $E\in\catPps^{-1}(Y/X)$ be a sheaf. Set $\calF\coloneqq \calH^{-1}(E)$. Then $\calF$ is quasi-isomorphic to $E[-1]$. In this case,
	\begin{align}
		\chi(\calF)= - d_{r+1}(E) \quad \text{and} \quad \sfch_1(\calF)=-\sfch_1(E)=\sum_{k=1}^r (d_k(E)-m_k d_{r+1}(E))C_k\ .
	\end{align}
	
	Let $\zeta\in \Q^{r+1}$ be a stability condition satisfying the inequality \eqref{eq:thetaineqA} and let $\omega \coloneqq \sum_{i=1}^{r}\,\zeta_i D_i$. Now, assume that $\ch_1(\calF)\cdot \omega\neq 0$ for both cases above. In addition, if $d_{r+1}(E) \neq 0$, we get
		\begin{align}
			\mu_\omega(\calF)=\frac{1}{\sum_{k=1}^{r} \zeta_k\, \big(m_k-\frac{d_k(E)}{d_{r+1}(E)}\big)}\ .
		\end{align}
		If the dimension vector $\bfd$ of $E$ belongs to $\zeta^\perp$, i.e., $\deg_\zeta(E)=0$, we get
		\begin{align}
			\mu_\omega(\calF)=\frac{1}{\sum_{k=1}^{r+1} m_k\, \zeta_k}\ .
		\end{align}		
		If $d_{r+1}(E)=0$, then $\mu_\omega(\calF)=0$.
\end{remark}

Fix $\zeta\in \Q^{r+1}$ a stability condition. We will show that under certain conditions on $\zeta$, we can define a $\Q$-polarization $\omega$ such that all objects in $\catPps(Y/X)$, which are sheaves and $\zeta$-semistable, are $\omega$-semistable as well, seen as objects of $\catCohps(Y)$.

\begin{proposition}\label{sheafwallcorA}
	Let $\zeta\in \Q^{r+1}$ be a stability condition satisfying the inequality \eqref{eq:thetaineqA} and let $\omega \coloneqq \sum_{i=1}^{r}\,\zeta_i D_i$. Let $E$ be a $\zeta$-(semi)stable object of $\catPps(Y/X)$ with dimension vector $\bfd\in \zeta^\perp$. Set
	\begin{align}
		\calF\coloneqq \begin{cases}
			\calH^0(E) & \text{if}\quad \displaystyle \sum_{i=1}^{r+1} m_i\, \zeta_i \geq 0\ , \\[15pt]
			\calH^{-1}(E) & \text{if}\quad \displaystyle \sum_{i=1}^{r+1} m_i\, \zeta_i < 0 \ .
		\end{cases}
	\end{align}
	Then $\calF$ is an $\omega$-(semi)stable coherent sheaf on $Y$, with slope
	\begin{align}
		\mu_\omega(\calF)=\frac{1}{\sum_{i=1}^{r+1}m_i\, \zeta_i} \ ,
	\end{align}
	if $\sum_{i=1}^{r+1}m_i\, \zeta_i\neq 0$, otherwise $\mu_\omega(\calF)=\infty$.
\end{proposition} 
\begin{proof}
	Proposition~\ref {sheafwallA} shows that $\calF$ is a coherent sheaf with proper support. If $\sum_{i=1}^{r+1}m_i\, \zeta_i$ vanishes, $\calF$ is a zero-dimensional sheaf. Hence $\calF$ is semistable with respect to any polarization $\omega$. 
	
	Suppose that $\sum_{i=1}^{r+1}m_i\, \zeta_i >0$. Let $\calF\twoheadrightarrow \calF''$ be a quotient in $\catCohps(Y)$, where $\calF''$ is a pure one-dimensional sheaf. Lemma~\ref {exseqlemm} shows that $E''\coloneqq \calF''$ belongs to $\catPps(Y/X)$, hence it satisfies\footnote{Here we use the notation $(\geq)$ and $(\leq)$ following \cite[Notation~1.2.5]{Huybrechts_Lehn_Moduli_2010}. For example, if in a statement the word ``(semi)stable'' appears together with relation signs ``$(\leq)$'', the statement encodes in fact two assertions: one about semistable sheaves and relation signs ``$\leq$'' and one about stable sheaves and relation signs ``$<$'', respectively.} 
	\begin{align}\label{eq:inequality}
		\deg_\zeta (E'')\ (\geq)\ 0\ .
	\end{align}
	By Formula~\eqref{eq:zetadegA}, we get
	\begin{align}
		\mu_\omega(\calF'')=\frac{1}{\sum_{i=1}^{r+1} m_i\, \zeta_i}\, \Big(\frac{\deg_\zeta(E'')}{\omega\cdot \sfch_1(E'')} +1\Big)\ .
	\end{align}
	Thus, the inequality \eqref{eq:inequality} is equivalent to 
	\begin{align}
		\mu_{\omega}(\calF'')\ (\geq)\ \frac{1}{\sum_{i=1}^{r+1} m_i
			\zeta_{i}} = \mu_{\omega}(\calF)\ . 
	\end{align}
	Hence $\calF$ is $\omega$-(semi)stable. 	
	
	Finally, assume that $\sum_{i=1}^{r+1}m_i\, \zeta_i <0$. Let $\calF'\hookrightarrow \calF$ be a subsheaf in $\catCohps(Y)$. Lemma~\ref {exseqlemm} shows that $E'\coloneqq \calF''[1]$ belongs to $\catPps(Y/X)$, hence it satisfies 
	\begin{align}\label{eq:inequality2}
		\deg_\zeta (E')\ (\leq)\ 0\ .
	\end{align}
	By arguing as before (and noticing that $\omega\cdot \sfch_1(E)<0$), we get 
	\begin{align}
	\mu_{\omega}(\calF'')\ (\leq)\ \frac{1}{\sum_{i=1}^{r+1} m_i
		\zeta_{i}} = \mu_{\omega}(\calF)\ . 
\end{align}
\end{proof}

\subsubsection{$\omega$-semistable sheaves on $Y$}

Now, we prove a converse of Proposition~\ref{sheafwallcorA}.

Let $\omega$ be a $\Q$-polarization of $Y$. Note that if a coherent sheaf $\calF\in \catCohps(Y)$ is $\omega$-stable, then it is pure. Below, we shall say that an $\omega$-stable sheaf is one-dimensional if and only if it has nontrivial one dimensional support, i.e., the zero sheaf is excluded. Moreover a curve on $Y$ will be a closed subscheme of pure dimension one. 

\begin{proposition}\label{finitewallF} 
	Let $\calF$ be an $\omega$-semistable properly supported pure coherent sheaf on $Y$ of dimension one with $\mu_\omega(\calF)>0$. Then, the one-term complex $E\coloneqq \calF$ of amplitude $[0,0]$ belongs to $\catPps(Y/X)$. Moreover, it is $\zeta$-semistable of degree zero, where
	\begin{align}
		\zeta_i&\coloneqq \langle \omega, C_i\rangle \quad\text{for}\quad i=1, \ldots, r\ ,\\[4pt]
		\zeta_{r+1}& \coloneqq \frac{1}{\mu_\omega(\calF)} - \sum_{i=1}^{r}\, m_i \langle \omega, C_i\rangle\ ,
	\end{align}
	and $\zeta\coloneqq (\zeta_1, \ldots, \zeta_{r+1})$.
\end{proposition}

\begin{proof}
	We assume that $\calF$ is stable. The strictly semistable case follows easily using Jordan-H\"older filtrations. 
	
	Let us denote by $E$ the one-term complex $\calF$ in $\catDb(\catCoh(Y))$. Since $\calF$ and $\scrO_{C_i}(-1)$ are $\omega$-stable sheaves such that $\mu_\omega(\calF)>0=\mu_\omega(\scrO_{C_i}(-1))$, we get $\Hom(\calF, \scrO_{C_i}(-1))=0$ for any $i=1, \ldots, r$. Thus, by Remark~\ref{rem:yoshioka}, $E\in \catPps(Y/X)$.
	
	Now, we prove that $E$ is $\zeta$-stable as a perverse coherent sheaf. First, thanks to Formula~\eqref{eq:zetadegA} the relation $\deg_\zeta(E)=0$ is immediate. Let $E\twoheadrightarrow G$ be a quotient in $\catPps(Y/X)$, not isomorphic to $E$. Using identity~\eqref{eq:zetadegA}, one obtains
	\begin{align}
		\deg_\zeta(G) =& d_{r+1}(G) \sum_{i=1}^{r+1} m_i\, \zeta_i + \omega \cdot 
		(\ch_1(\calH^{-1}(G)) - \ch_1(\calH^0(G)))\\
		=& \big( \dim\, H^0(\calH^0(G)) +  \dim\, H^1(\calH^{-1}(G)) \big)\sum_{i=1}^{r+1} m_i\, \zeta_i + \omega \cdot 
		(\ch_1(\calH^{-1}(G)) - \ch_1(\calH^0(G)))\ ,
	\end{align}
	where
	\begin{align}
		\sum_{i=1}^{r+1}m_i\zeta_i = \frac{1}{\mu_\omega(\calF)} >0\ .
	\end{align}
	This further implies
	\begin{align}
		\deg_\zeta(G)  \geq \dim\, H^0(\CH^0(G)) \sum_{i=1}^{r+1} m_i\, \zeta_i - \omega\cdot \ch_1(\calH^0(G))\ .
	\end{align}
	Since $G$ is a perverse coherent sheaf, $H^1(\calH^0(G))=0$, hence
	${\rm dim}\, H^0(\calH^0(G)) = \chi(\calH^0(G))$. 
	Therefore the previous inequality yields 
	\begin{align}
		\deg_\zeta(G) & \geq   \mu_\omega(\calF)^{-1}\chi(\calH^0(G)) - \omega\cdot \ch_1(\calH^0(G))\\
		& =   (\omega\cdot \ch_1(\calH^0(G)))\left(\mu_\omega(\calH^0(G))\mu_\omega(\calF)^{-1} -1\right)>0=\deg_\zeta(E) \ ,
	\end{align}
	where last inequality follows from the fact that $\calH^0(G)$ is a quotient of $\calF$ and $\calF$ is $\omega$-stable. 
\end{proof}

\begin{remark}
	Note that the stability condition $\zeta$ defined in the proposition above is such that
	\begin{align}
		\sum_{i=1}^{r+1}m_i\, \zeta_i =\frac{1}{\mu_\omega(\calF)}>0\ .
	\end{align}
\end{remark}

By analogous arguments as those in the proof of Proposition~\ref{finitewallF}, one obtain the following result for zero-dimensional sheaves.
\begin{proposition}\label{zerosheaves} 
	Let $\zeta \in \Q^{r+1}$ be a stability condition such that 
	\begin{align}
		\sum_{i=1}^{r+1} m_i\, \zeta_i =0\ .
	\end{align}
	Let $\calF$ be a zero-dimensional sheaf on $Y$. Then, the one-term complex $E\coloneqq \calF$ of amplitude $[0,0]$ belongs to $\catPps(Y/X)$. Moreover, it is $\zeta$-semistable of degree zero.
\end{proposition}

We finish this section with a characterization of the first Chern class of properly supported stable sheaves.\footnote{A similar characterization is used in \cite[\S2.2, Lemma~4]{IUU_Stability}.}
\begin{proposition}\label{finitewallA} 
	Let $\omega$ be a $\Q$-polarization of $Y$.  Let $\calF$ be an $\omega$-stable one-dimensional sheaf on $Y$ with proper support. Then there exist a (connected) subquiver $\calQ(\calF)=(\calQ(\calF)_0, \calQ(\calF)_1)$ of the type ADE quiver $\calQ^{\mathsf{fin}}$ and a root $\alpha_\calF$ for the Lie algebra $\frakg_{\calQ(\calF)}$ of the form
	\begin{align}\label{eq:rootalphaF}
		\alpha_\calF\coloneqq \sum_{\imath\in \calQ(\calF)_0}s_\imath\alpha_\imath \ ,
	\end{align}
	with $s_\imath>0$ for any vertex $\imath$ of $\calQ(\calF)$, such that
	\begin{align}\label{eq:stabchA} 
		\sfch_1(\calF) = \sum_{\imath\in \calQ(\calF)_0}s_\imath C_\imath\ .
	\end{align}
	Moreover, 
	\begin{align}\label{eq:vanextA}
		\Ext^1_Y(\calF, \calF)=0\ .
	\end{align}
\end{proposition} 
\begin{proof}
	Since $\calF$ is $\omega$-stable and purely one-dimensional, we have
	\begin{align}
		\sfch_1(\calF)= \sum_{i=1}^r\, m_i C_i\ ,
	\end{align}
	with $m_i\geq 0$ for $1\leq i \leq r$, and its set-theoretic support is connected. Thus, we can associate with $\sfch_1(\calF)$ a type ADE quiver $\calQ(\calF)$ (which is a subquiver of the quiver $\calQ^{\mathsf{fin}}$), whose vertices $\imath$ correspond to $C_i$ with $m_i>0$ in $\sfch_1(\calF)$. Therefore, we can rewrite $\sfch_1(\calF)$ as follows:
	\begin{align}
		\sfch_1(\calF)=\sum_{\imath \in \calQ(\calF)_0}s_\imath C_\imath \ .
	\end{align}
	By the Riemann-Roch theorem and the McKay correspondence (cf.\ Proposition~\ref{prop:McKay}-(\ref{Mckay-(finite)})), one has 
	\begin{align}
		\chi(\calF,\calF)  = - \sfch_1(\calF)^2 = (\alpha_\calF, \alpha_\calF)_{\calQ(\calF)} \ ,
	\end{align}
	where 
	\begin{align}\label{eq:alphaF}
		\alpha_\calF\coloneqq \sum_{\imath \in \calQ(\calF)_0}s_\imath\alpha_\imath
	\end{align}
	is the element of the root lattice of the Lie algebra $\frakg_{\calQ(\calF)}$ corresponding to $\sfch_1(\calF)$, and $(-, -)_{\calQ(\calF)}$ denotes its Cartan-Killing form. Since type ADE root lattices are even, we get  
	\begin{align}
		\chi(\calF, \calF)\in 2\, \Z_{>0}  \ .
	\end{align}
	Since $\calF$ is $\omega$-stable, $\dim \Ext^0_Y(\calF,\calF)=1$. Using Serre duality, $\dim \Ext^2_Y(\calF, \calF)=1$ as well. Hence 
	\begin{align}
		\chi(\calF, \calF)=2-\dim \Ext^1_Y(\calF,\calF)\leq 2\ .
	\end{align}
	Thus, $\chi(\calF, \calF)$ is exactly two, hence $\alpha_\calF$ is a root.\footnote{In a type ADE root lattice, any element $\beta$ such that $(\beta, \beta)=2$ is a root.} Therefore, the vanishing result \eqref{eq:vanextA} holds. 
\end{proof}

\begin{remark}
	Note that if $\calF$ is a $\omega$-stable compactly supported one-dimensional sheaf on $Y$ with $\sfch_1(\calF)=C_i$ for some $1\leq i\leq r$, then necessarily $\calF\simeq \scrO_{C_i}(d)$ for some $d\in \Z$.
\end{remark}

Recall that we can associate to a coherent sheaf $\calF$ on $Y$ its support $Z_0$ and its Fitting support $Z$. As stated in \cite[\href{https://stacks.math.columbia.edu/tag/0C3C}{Tag~0C3C}, Lemma~31.9.3]{stacks-project}, one has, $Z_0\subset Z$, $Z_{\mathsf{red}}=Z_0$, and there exists a coherent sheaf $\calG$ on $Z$ such that
\begin{align}
	i_\ast(\calG)=\calF \ ,	
\end{align}
where $i$ denote the embedding of $Z$ into $Y$. Moreover, if $\calF$ is a pure one-dimensional sheaf, $Z$ is a representative of $\sfch_1(\calF)$. \footnote{We do not distinguish between $\calF$ and $\calG$ when it is clear from the context.}

Now, if $\calQ(\calF)$ is of type A, the only root of the form \eqref{eq:rootalphaF} is its highest root. Thus $m_\imath=1$ for any vertex $\imath$ of $\calQ(\calF)$. Correspondingly, the Fitting support of $\calF$ is reduced. Thus, in the type A case, we have the following:
\begin{corollary}\label{cor:finitewallA}
	Let $\calQ^{\mathsf{fin}}$ be a type A quiver. Let $\omega$ be a $\Q$-polarization of $Y$.  Let $\calF$ be an $\omega$-stable one-dimensional sheaf on $Y$ with proper support. Then
	\begin{align}\label{eq:Cij}
		\sfch_1(\calF)=C_i+\cdots + C_j\eqqcolon C_{i, j}
	\end{align}
	for some $1\leq i \leq j\leq r$. In particular, $\calF$ is scheme-theoretically supported on the reduced divisor $C_{i, j}$.
\end{corollary}

\subsection{Categorical McKay correspondence for semistable objects}\label{mckaysect}

Recall that $\calQ^{\mathsf{fin}}$ is the finite Dynkin diagram of type ADE associated with the finite group $G$ and $\calQ$ denotes the corresponding affine Dynkin diagram. 

Fix a polarization $\omega$ on $Y$. Denote by $\catCoh_\mu^{\omega\textrm{-}\mathsf{ss}}(Y)$ the abelian category of $\omega$-semistable one-dimensional sheaves on $Y$ of slope $\mu$. By abuse of notation, we denote by $\catCoh_\infty^{\omega\textrm{-}\mathsf{ss}}(Y)$ the abelian category of zero-dimensional  sheaves on $Y$.

Let $\zeta\in \Q^{r+1}$ be a stability condition. Denote by $\sfRep_0^{\zeta\textrm{-}\mathsf{ss}}(\Pi_\calQ)$ the abelian category of $\zeta$-semistable finite-dimensional representations of $\Pi_{\calQ}$ of zero $\zeta$-slope. 

Thus, our first main result is that the tilting functor ``preserves'' semistability:
\begin{theorem}\label{categmckay} 
	Let $\omega$ be a $\Q$-polarization of $Y$ and $\mu\in \Q_{>0}\cup\{\infty\}$. Set
	\begin{align}
		\zeta_i\coloneqq \langle \omega, C_i\rangle
	\end{align}
	for $1\leq i \leq r$, and
	\begin{align}
		\zeta_{r+1}\coloneqq\begin{cases}
		\displaystyle	\frac{1}{\mu}  - \sum_{i=1}^{r}m_i\, \zeta_i\ & \text{if } \mu\neq \infty\ ,\\[8pt]
		\displaystyle	- \sum_{i=1}^{r} m_i\, \zeta_i &\text{otherwise}\ .
		\end{cases}
	\end{align}
	Then the tilting functor \eqref{eq:cateqB} yields an equivalence
	\begin{align}
		\catCoh_\mu^{\omega\textrm{-}\mathsf{ss}}(Y) \simeq \sfRep_0^{\zeta\textrm{-}\mathsf{ss}}(\Pi_\calQ) \ .
	\end{align}
\end{theorem}

\begin{proof}
	Let $\calF$ be a $\omega$-(semi)stable one-dimensional sheaf with $\mu_\omega(\calF)=\mu>0$.
	Let $E \coloneqq \calF$ be a the one-term complex concentrated in amplitude $[0,0]$ associated to $\calF$.
	Then Proposition~\ref{finitewallF} shows that $E$ belongs to $\catPps(Y/X)$, that it is $\zeta$-(semi)stable and that it has degree zero.
	
	Viceversa, let $E\in \catPps^0(Y/X)$ be $\zeta$-(semi)stable. Since $\zeta_i>0$ for $i=1, \ldots, r$, and 
	\begin{align}
		\sum_{i=1}^{r+1}m_i\, \zeta_i >0\ .
	\end{align}
	we can apply Proposition~\ref {sheafwallA}, which shows that $\calF\coloneqq \calH^0(E)$ is a $\omega$-(semi)stable one-dimensional pure sheaf on $Y$ with slope $\mu$. One can argue similarly if $\mu=\infty$ and get the assertion.
\end{proof}

\subsection{Chains of stable sheaves and type A Dynkin quiver}\label{s:chain} 

Let $G$ be $\Z_{N+1}$ for $N\in \Z$, $N\geq 1$ (hence, $r=N$). In this section, we will provide a finer characterization of $\omega$-(semi)stable sheaves of positive slope.  

\subsubsection{Characterization of stable sheaves}

We start by recalling the following version of Grothendieck duality (cf.\ \cite{Conrad_duality}).
\begin{lemma}\label{lem:extlemm}
	Let $S$ be a smooth quasi-projective surface, and let $\calF,\calG$ be two purely one dimensional coherent sheaves on $S$ with proper support. Suppose $\calG$ is scheme theoretically supported on a divisor $D$ so that $\calF\otimes \scrO_D$ is zero-dimensional. 
	Then Grothendieck duality for the closed embedding $D \hookrightarrow S$ yields a functorial isomorphism 
	\begin{align}
		\Ext^1_S(\calG,\calF) \simeq \Hom_D(\calG, \calF\otimes \scrO_D(D))\ . 
	\end{align}
\end{lemma} 

\begin{proof}
	This follows for example from Grothendieck duality for the closed embedding $D\hookrightarrow S$ using the fact that its relative dualizing complex is $O_{D_2} (D_2)[-1]$.
\end{proof}

Let $\calF$ be a one-dimensional pure coherent sheaf on $Y$ with 
\begin{align}
	\sfch_1(\calF)=C_{i,j}
\end{align}
for some $1\leq i < j\leq N$. For $i\leq \ell\leq j$, let $\calT_{\ell,j}$ be the maximal zero-dimensional subsheaf of $\calF\otimes \scrO_{C_{\ell, j}}$ and set $\calG_{\ell,j}\coloneqq \calF\otimes \scrO_{C_{\ell, j}}/\calT_{\ell,j}$. By construction, for each $i\leq \ell \leq j-1$ there is an exact sequence
\begin{align}\label{eq:symmexseqA} 
	0 \longrightarrow \calF_{i,\ell} \longrightarrow \calF \longrightarrow \calG_{\ell+1, j} \longrightarrow 0\ ,
\end{align}
where $\calF_{i,\ell}$ and $\calG_{\ell+1,j}$ are one-dimensional pure sheaves with 
\begin{align}
	\sfch_1(\calF_{i,\ell}) = C_{i,\ell} \quad\text{and}\quad \sfch_1(\calG_{\ell+1,j}) = C_{\ell+1,j}\ .
\end{align}
In particular they are scheme-theoretically supported  on $C_{i,\ell}$ and $C_{\ell+1,j}$ respectively.  By applying Grothendieck duality for the closed embedding $C_{\ell+1,j} \hookrightarrow Y$ (cf.\ Lemma~\ref{lem:extlemm}), let 
\begin{align}\label{eq:phimapA}
	\phi_{\ell+1}\colon \calG_{\ell+1}\to \calF_{i,\ell}\otimes \scrO_{C_{\ell+1,j}}(C_{\ell+1,j})
\end{align}
be the morphism corresponding to the extension class of \eqref{eq:symmexseqA}. 

\begin{lemma}\label{structlemma} 
	For any $i\leq \ell\leq j$ there is a commutative diagram with exact rows and columns 
	\begin{align}\label{eq:symmdiagA}
		\begin{tikzcd}[ampersand replacement = \&]
			\& 0\arrow[d] \& 0 \arrow[d] \& \\
			\& \calF_{i,\ell-1} \arrow{r}{\bf 1} \arrow[d] \& \calF_{i,\ell-1}\arrow[d] \& \\
			0\arrow[r] \& \calF_{i,\ell} \arrow[r] \arrow{d}{f_\ell} \& \calF \arrow[r] \arrow[d] \&  \calG_{\ell+1,j} 
			\arrow{d}{\bf 1}\arrow[r] \& 0\\
			0\arrow[r] \& \calF_\ell \arrow{r}{g_\ell}\arrow[d] \& \calG_{\ell,j}\arrow[r] \arrow[d] \& \calG_{\ell+1,j} \arrow[r]\& 0 \\
			\& 0 \& 0 \&
		\end{tikzcd}
	\end{align}	
	where $\calF_\ell \simeq \scrO_{C_\ell}(e_\ell)$ for some $e_\ell \in \Z$, and,
	by convention, $\calF_{i,i-1}$ and $\calG_{j,j+1}$ are identically zero. In particular there are isomorphisms $\calF_{i,i}\simeq \calF_i$ and $\calG_{j,j}\simeq \calF_j$. 
	
	Moreover, the morphism \eqref{eq:phimapA} determines uniquely a second morphism
	\begin{align}\label{eq:psimapA}
		\psi_{\ell+1}\colon \calF_{\ell+1}\to \calF_\ell\otimes \scrO_{C_{\ell+1,j}}(C_{\ell+1,j}) \ ,
	\end{align}
	which fits in a commutative diagram 
	\begin{align}\label{eq:symmdiagB} 
		\begin{tikzcd}[ampersand replacement = \&]
			\calF_{\ell+1} \arrow{d}{g_{\ell+1}} \arrow{rr}{\psi_{\ell+1}} \& \& \calF_\ell\otimes 
			\scrO_{C_{\ell+1,j}}(C_{\ell+1,j}) \\
			\calG_{\ell+1,j}  \arrow{rr}{\phi_{\ell+1}} \& \& \calF_{i,\ell}\otimes \scrO_{C_{\ell+1,j}}(C_{\ell+1,j}) \arrow{u}{f_\ell\otimes {\bf 1}}
		\end{tikzcd}\ .
	\end{align}
\end{lemma} 

\begin{proof} 
	By construction there is a natural epimorphism $\calG_{\ell, j} \twoheadrightarrow \calG_{\ell+1, j}$ since the support condition implies $\Hom_Y(\calF_{i, \ell-1}, \calG_{\ell+1, j})=0$. Then diagram \eqref{eq:symmdiagA} follows from the snake lemma. 
	
	Since $\calG_{\ell+2, j}$ and $\calF_{i, \ell}$ have disjoint support, one obtains a natural isomorphism 
	\begin{align}
		\Hom_Y(\calG_{\ell+1, j}, \calF_{i, \ell}\otimes \scrO_{C_{\ell+1, j}}(C_{\ell+1, j})) \simeq \Hom_Y(\calF_{\ell+1}, \calF_{i, \ell}\otimes \scrO_{C_{\ell+1, j}}(C_{\ell+1, j}))\ .
	\end{align}
	Again, since $\calF_{\ell+1}$ and $\calF_{i, \ell-1}$ have disjoint support, one further has a natural isomorphism 
	\begin{align}
		\Hom_Y(\calF_{\ell+1}, \calF_{i, \ell}\otimes \scrO_{C_{\ell+1, j}}(C_{\ell+1, j}))\simeq \Hom_Y(\calF_{\ell+1}, \calF_{\ell}\otimes \scrO_{C_{\ell+1, j}}(C_{\ell+1, j}))\ .
	\end{align} 	
	This implies the second part of Lemma~\ref {structlemma}.
\end{proof} 

In conclusion, up to isomorphism,  $\calF$ determines uniquely a sequence of line bundles $(\calF_i, \ldots, \calF_j)$ on the curves $C_i, \ldots, C_j$ of degrees $(e_i, \ldots, e_j)$ respectively. Note that 
\begin{align}\label{eq:tensisomA}
	\calF_\ell\otimes \scrO_{C_{\ell+1,j}}(C_{\ell+1,j}) \simeq \calF_\ell\vert_{p_\ell}\simeq \scrO_{p_\ell}
\end{align}
where $p_\ell$ is the transverse intersection point between $C_\ell$ and $C_{\ell+1}$. Then, by the Grothendieck duality for the closed embedding $C_{\ell+1,j} \hookrightarrow Y$ (cf.\ Lemma~\ref{lem:extlemm}), we have
\begin{align}
	\Ext^1_Y(\calF_\ell, \calF_m) \simeq \begin{cases}
		\C & \text{for } \ell = m+1\ , \\
		0 & \text{otherwise}\ .
	\end{cases}
\end{align}
Therefore, $\calF$ admits a recursive construction through a sequence of extensions 
\begin{align}
	0 \longrightarrow \calF_{\ell} \longrightarrow \calG_{\ell, j} \longrightarrow \calG_{\ell+1, j} \longrightarrow 0\ ,
\end{align}
where $\calG_{j,j} \simeq \calF_{j}$. At each step the associated extension class corresponds to the morphism \eqref{eq:phimapA}, or, equivalently, \eqref{eq:psimapA}. 

\begin{lemma}\label{finitewallAB} 
	Let $\omega$ be a $\Q$-polarization of $Y$. Let $\calF$ be a one-dimensional pure coherent sheaf on $Y$ with 
	\begin{align}
		\sfch_1(E)=C_{i,j}
	\end{align}
	for some $1\leq i < j\leq N$. If $\calF$ is $\omega$-stable, all morphisms $\phi_{\ell+1}$ are nonzero. Moreover, let $\calL$ denote the following line bundle on $Y$:
	\begin{align}
		\calL \coloneqq \otimes_{\ell=i}^{j-1}\scrO_Y((-1-e_\ell)D_\ell)\otimes \scrO_Y(-e_j D_j)\ .
	\end{align}
	Then $\calF\otimes \calL\simeq \scrO_{C_{i,j}}$. 
\end{lemma}
\begin{proof}
	If $i=j$ there is nothing to prove. 
	
	Suppose $i<j$. Then the first claim is clear. 
	Indeed, if $\phi_{\ell+1}$ is zero for some $i\leq\ell\leq j-1$, the exact sequence \eqref{eq:symmexseqA} splits, and this contradicts stability. In order to prove the second claim, note that the sequence of line bundles $\calF_\ell'$ corresponding to $\calF'\coloneqq \calE\otimes \calL$ has degrees  $(-1,-1, \ldots, -1, 0)$. Moreover, since all successive morphisms $\phi_{\ell+1}'$ with $i\leq \ell < j$ are nonzero, each exact sequence 
	\begin{align}
		0\longrightarrow \scrO_{C_\ell}(-1) \longrightarrow \calG_{\ell,j}' \longrightarrow \calG_{\ell+1,j}' \longrightarrow 0
	\end{align}
	has a nonzero extension class in 
	\begin{align}
		\Ext^1_Y(\calG_{\ell+1,j}',\scrO_{C_\ell}(-1)) \simeq \C\ .
	\end{align}
	Then an easy inductive argument shows that each quotient $\calG_{\ell,j}\simeq \scrO_{C_{\ell,j}}$ for all $i\leq \ell\leq j$. 
\end{proof}

\begin{proposition}\label{finitewallB}
	Let $\omega$ be a $\Q$-polarization of $Y$. Set
	\begin{align}
		\zeta_i\coloneqq \langle \omega, C_i\rangle
	\end{align}
	for $1\leq i \leq N$. Let $\calF$ be a one-dimensional pure coherent sheaf on $Y$ with 
	\begin{align}\label{eq:topinvA} 
		\sfch_1(\calF)=C_{i,j}
	\end{align}
	for some $1\leq i < j\leq N$. If $\calF$ is $\omega$-stable, we have
	\begin{align}\label{eq:tailsheafA}
		\zeta_i \mu_\omega(\calF) -2 < e_i < \zeta_i\mu_\omega(\calF)-1 \quad\text{and}\quad 
		\zeta_j \mu_\omega(\calF) -1 < e_j < \zeta_j\mu_\omega(\calF)\ ,
	\end{align}
	and for any $i<\ell<j$  
	\begin{align}\label{eq:midsheafA} 
		\zeta_\ell \mu_\omega(\calF) -2 < e_\ell < \zeta_\ell\mu_\omega(\calF)\ . 
	\end{align}
\end{proposition}
\begin{proof}
	First suppose $i<\ell<j$. Let $z\in H^0(\scrO_{C_\ell}(1))$ be a defining section of $p_\ell$ and let $\zeta\colon \calF_\ell \to \scrO_{C_\ell}(e_\ell+1)$ be multiplication by $z$. Using the isomorphism \eqref{eq:tensisomA}, one  has 
	\begin{align}
		\zeta \circ \psi_{\ell+1} = 0\ . 
	\end{align}
	Using Lemma~\ref{lem:extlemm}, this implies that the morphism $\zeta\colon \calF_\ell \to \scrO_{C_\ell}(e_{\ell}+1)$ extends to a morphism $\calG_{\ell, j}\to 
	\scrO_{C_\ell}(e_{\ell}+1)$ which fits in the commutative diagram
	\begin{align}
		\begin{tikzcd}[ampersand replacement = \&]
			\calF_{\ell} \arrow[r] \arrow{d}{\zeta} \& \calG_{\ell,j} \arrow[d] \\
			\scrO_{C_\ell}(e_\ell+1) \arrow{r}{\bf 1} \&   \scrO_{C_\ell}(e_\ell+1)
		\end{tikzcd}\ .
	\end{align}	
	Since $\calG_{\ell,j}$ is a quotient of $\calF$, one obtains a nonzero morphism $\calF\to  \scrO_{C_\ell}(e_\ell+1)$. The target is $\omega$-stable of slope 
	\begin{align}
		\mu_\omega(\scrO_{C_\ell}(e_\ell+1))=\frac{e_\ell+2}{\zeta_\ell}\ .
	\end{align}
	Therefore one obtains 
	\begin{align}
		\zeta_\ell\mu_\omega(\calF) < e_\ell+2\ .
	\end{align}
	Next, note that $\calF_{\ell+1}$ is a quotient of $\calF_{i,\ell+1}$. Again, let $\eta\colon \scrO_{C_{\ell+1}}(e_{\ell+1}-1)\hookrightarrow \calF_{\ell+1}$ be multiplication by $z$, which is obviously injective. Then note that 
	\begin{align}
		\psi_{\ell+1}\circ \eta =0\ .
	\end{align}
	Then Lemma~\ref{lem:extlemm} shows that $\eta$ lifts to a morphism 
	$\scrO_{C_{\ell+1}}\to \calF_{i,\ell+1}$ which fits in the commutative diagram
	\begin{align}
		\begin{tikzcd}[ampersand replacement = \&]
			\scrO_{C_{\ell+1}}(e_{\ell+1}-1)\arrow{d}{\eta}  \arrow{r}{\bf 1} \&  
			\scrO_{C_{\ell+1}}(e_{\ell+1}-1)\arrow[d] \\ 
			\calF_{\ell+1} \arrow[r]  \& \calF_{i, \ell+1}
		\end{tikzcd}\ .
	\end{align}	
	Since $\calF_{i, \ell+1}$ is a subsheaf of $\calF$, one obtains 
	\begin{align}
		e_{\ell+1} < \zeta_{\ell+1} \mu_\omega(\calF)\ .
	\end{align}
	The proof of \eqref{eq:tailsheafA} is completely analogous. 	
\end{proof}

\begin{corollary}
	Let $\omega$ be a $\Q$-polarization of $Y$. Let $\calF$ be a $\omega$-stable one-dimensional pure coherent sheaf on $Y$ with positive slope and 
	\begin{align}
		\sfch_1(E)=C_{i,j}
	\end{align}
	for some $1\leq i < j\leq N$. Then
	\begin{align}
			e_\ell \geq -1\ , \quad \ i \leq \ell <j-1\ , \quad e_j \geq 0\ .
	\end{align}
\end{corollary}
%
%\begin{corollary}\label{finitewallcorC} 
%	For fixed $\omega$ and $\mu\in \Q$ there are finitely many isomorphism classes of properly supported $\omega$-stable pure dimension one sheaves of slope $\mu$. 
%\end{corollary} 

\begin{corollary}\label{finitewallcorD} 
	Let $\omega= \sum_{i=1}^{N}D_i$ be the symmetric polarization. Let $\calF$ be an $\omega$-stable pure one-dimensional sheaf on $Y$ with proper support such that $\mu_\omega(E)=1$. Then $\calF\simeq \scrO_{C_i}$ for some $1\leq i \leq N$. 
\end{corollary}

\begin{proof}
	Suppose $\ch_1(\calF)=C_{i,j}$ for some $1\leq i < j \leq N$. Under the present assumptions, the first inequality in \eqref{eq:tailsheafA} implies that 
	\begin{align}
		-2< e_i <-1\ ,
	\end{align}
	which is a contradiction. Hence one must have $i=j$ and $e_j=0$. 
\end{proof}

\begin{remark}\label{rem:stable-dimension-vector}
	By Proposition~\ref{finitewallF}, a $\omega$-stable properly supported sheaf $\calF$ on $Y$ of positive slope $\mu$ and first Chern class $\ch_1(\calF)=C_{i, j}$ corresponds to a perverse coherent sheaf $E\in \catPps(Y/X)$, which is semistable with respect to $\zeta=(\zeta_1, \ldots, \zeta_N, \zeta_{N+1})$, with $\zeta_i\coloneqq \langle \omega, C_i\rangle$ and $\zeta_{N+1}\coloneqq 1/\mu -\sum_{k=1}^N \zeta_k$, and degree zero. Moreover, thanks to Lemma~\ref{dimvectlemm}, we get
	\begin{align}
		d_{N+1}(E)&=\chi(\scrO_Y, \calF)= \mu \sum_{k=i}^j \zeta_k\ , \\[2pt]
		d_k(E)&=\begin{cases}
			d_{N+1}(E)-1 & \text{for } i\leq k\leq j\ ,\\
			d_{N+1}(E) & \text{otherwise} \ ,
		\end{cases}
	\end{align}
	for $1\leq k\leq N$. If we set $\delta=(1,\ldots, 1)$ and we denote by $\alpha_k$ the $k$-th coordinate vector of $\Z^{N+1}$, we obtain
	\begin{align}
		\bfd(E)=\mu \sum_{k=i}^j \zeta_k \delta-\sum_{k=i}^j \alpha_k\ .
	\end{align}
\end{remark}

\subsubsection{Chains and semistable sheaves}\label{chainsect}

Let $\omega$ be a $\Q$-polarization of $Y$. Set
\begin{align}
	\zeta_i\coloneqq \langle \omega, C_i\rangle
\end{align}
for $1\leq i \leq N$. Fix $\mu\in \Q_{>0}$. We denote by $\sfS_\mu$ the set of isomorphism classes of $\omega$-stable properly supported one-dimensional sheaves of slope $\mu$. Given a representative $\calF_\alpha$ of an equivalence class $\alpha\in \sfS_\mu$, Proposition~\ref{finitewallA} and Corollary~\ref{cor:finitewallA} guarantee that
\begin{align}
	\sfch_1(\calF_\alpha) = C_{i_\alpha, j_\alpha} 
\end{align}
for integers $1 \leq i_\alpha \leq j_\alpha \leq N$.
\begin{proposition}\label{stabextlemm} 
	Let $\calF_\alpha, \calF_\beta$ be representatives of the equivalence classes $\alpha, \beta\in \sfS_\mu$ respectively and let 
	\begin{align}
		\sfch_1(\calF_\alpha) = C_{i_\alpha, j_\alpha}\quad \text{and} \quad \sfch_1(\calF_\beta) = C_{i_\beta, j_\beta}\ .
	\end{align}
	for integers $1 \le i_\alpha \le j_\alpha \le N$ and $1 \le i_\beta \le j_\beta \le N$. Then one has:
	\begin{enumerate}[label=(\roman*)]\itemsep=0.2cm
		\item \label{stabextlemm:iso_classes} If $i_\alpha = i_\beta$ or $j_\alpha = j_\beta$  then $\alpha = \beta$, hence $\calF_\alpha \simeq \calF_\beta$. In particular any isomorphism class $\alpha$ is uniquely determined by the pair $(i_\alpha, j_\alpha)$.
		
		\item \label{stabextlemm:ext} Moreover,
		\begin{align}\label{eq:extableB} 
		\Ext^p_Y(\calF_\alpha, \calF_\beta) \simeq  \begin{cases}
		\C & \text{if } p=0\ \text{and}\ \alpha = \beta \ , \\[4pt]
		\C & \text{if } p=1\ \text{and either }\ i_\beta= j_\alpha+1\ \text{or}\ i_\alpha = j_\beta+1\ ,\\[4pt]
		0 & \text{otherwise}\ .
		\end{cases}
		\end{align}
	\end{enumerate}
\end{proposition} 

\begin{proof}
	Since all $\calF_\alpha$ are stable of the same slope one has
	\begin{align}\label{eq:extableA}
		\Hom_Y(\calF_\alpha, \calF_\beta) = \begin{cases}
		\C& \text{if } \alpha = \beta\ , \\[4pt] 
		0 & \text{otherwise}\ .
		\end{cases}
	\end{align}
	Hence, using Serre duality, 
	\begin{align}\label{eq:chitableA}
		\chi(\calF_\alpha, \calF_\beta) =\begin{cases}
		2-\dim\Ext^1_Y(\calF_\alpha, \calF_\beta) & \text{if } \alpha = \beta\ , \\[4pt] 
		-\dim\Ext^1_Y(\calF_\alpha, \calF_\beta)  & \text{otherwise}\ .
		\end{cases}
	\end{align}
	Moreover, equation~\eqref{eq:vanextA} shows that $\Ext^1_Y(\calF_\alpha, \calF_\beta)=0$ if $\alpha=\beta$. However, by the Riemann-Roch theorem, 
	\begin{align}\label{eq:chiRR}
		\chi(\calF_\alpha, \calF_\beta) = -C_{i_\alpha, j_\alpha}\cdot C_{i_\beta, j_\beta}\ ,
	\end{align}
	where $C_{i,j}$ is the divisor introduced in equation~\eqref{eq:Cij} for $1\leq i, j\leq N-1$.
	
	For any pairs $(i,j)$, $(\ell,m)$ with $i\leq j$ and $\ell\leq m$ one has the following cases: 
	\begin{enumerate}\itemsep0.2cm
		
		\item If $i=\ell$ and $j=m$, then $C_{i, j}\cdot C_{\ell,m} =-2$. 
		
		\item \label{item:(ii)} If $i=\ell$ and $j\neq m$ or $j=m$ and $i\neq \ell$, then $C_{i, j}\cdot C_{\ell,m} =-1$.
		
		\item If $\ell=j+1$ or $i=m+1$, then $C_{i, j}\cdot C_{\ell,m} =1$. In this case the divisors $C_{i,j}$ and $C_{\ell,m}$ will be called \textit{linked}.
		
		\item In all other cases, $C_{i, j}\cdot C_{\ell,m} =0$.
	\end{enumerate}

	Under the stated conditions, the divisors $C_{i_\alpha, j_\alpha}$ and $C_{i_\beta, j_\beta}$ cannot satisfy case \eqref{item:(ii)} above, since that would lead to a contradiction with equations \eqref{eq:chitableA} and \eqref{eq:chiRR}.	This proves statement \ref{stabextlemm:iso_classes}. 
	
	For the remaining cases, Equations \eqref{eq:extableA}, \eqref{eq:chitableA} and \eqref{eq:chiRR} yield: 
	\begin{align}\label{eq:extableC}
		\Ext^1_Y(\calF_\alpha, \calF_\beta) \simeq \begin{cases}
		\C & \text{if } i_\beta=j_\alpha+1\ \text{or}\ i_\alpha=j_\beta+1, \\
		0 & \text{otherwise} \ . 
		\end{cases}
	\end{align}
\end{proof}

Note that for $\alpha, \beta\in \sfS_\mu$, one can either have $\Ext^1(\calF_\alpha, \calF_\beta)=\C$ or $\Ext^1(\calF_\alpha, \calF_\beta)=0$. This will allow us to divide $\sfS_\mu$ into disjoint subsets and order the equivalence classes in each subset, as we shall explain now. First, we introduce the following notion.
\begin{definition}\label{chaindef} 
	An ordered sequence ${\boldsymbol \alpha}=({\alpha_1}, \ldots,{\alpha_s})\in {\sf S}_\mu$  will be called a \textit{chain} if and only if 
	\begin{enumerate}[label=(\roman*)]\itemsep0.2cm
	
		\item \label{chaindef:concatenation} $j_{\alpha_t}=i_{\alpha_{t+1}}$ for all $1\leq t \leq s-1$, and 	
			
		\item it is not a strict subsequence of any other ordered sequence of elements of ${\sf S}_\mu$ satisfying Proposition~\ref{stabextlemm}-\ref{chaindef:concatenation}.
		
	\end{enumerate}
	For simplicity, set  $\Delta_t = C_{i_{\alpha_t}, j_{\alpha_t}}$, $1\leq t \leq s$. 
\end{definition}

\begin{example}\label{chainexample} 
	Suppose $\omega= \sum_{i=1}^{N} D_i$ is the symmetric polarization and let $\mu=1$. By Corollary~\ref {finitewallcorD}, a slope one properly supported sheaf $\calF$ on $Y$ is $\omega$-stable if and only if $\calF\simeq \scrO_{C_i}$ for some $1\leq i \leq N$. Therefore in this case the set of stable objects is $\{\scrO_{C_i}\,|\, 1\leq i \leq N\}$ and they form a single chain of length $N$. 
\end{example}

\begin{remark}\label{partitionchains}
	Proposition~\ref{stabextlemm} implies that there is a unique partition into pairwise disjoint subsets 
	\begin{align}
		\sfS_\mu = \bigcup_{c=1}^{C_\mu} \{\, \alpha_{c,1}, \ldots, 
	\alpha_{c,s_c}\, \}
	\end{align}
	such that for each $1\leq c\leq C_\mu$ the ordered sequence 
	\begin{align}	
	\balpha_c \coloneqq (\alpha_{c,1}, \ldots, \alpha_{c, s_c})
	\end{align}
	is a chain for all $1\leq c \leq C_\mu$. Let $\lambda_\mu\coloneqq (s_1, s_2, \ldots, s_{C_\mu})$ denote the induced {\it ordered} partition of $N_\mu$.
\end{remark}

In order to formulate some further consequences of Proposition~\ref {stabextlemm}, let us introduce the following definition:

\begin{definition}\label{JHchain}	
	An $\omega$-semistable properly supported sheaf $\calF$ on $Y$ will be said to belong to a chain $\balpha_c$, i.e. $\calF\in \balpha_c$, if and only if each of its Jordan-Hölder factors belongs to one of the isomorphism classes $\alpha_{c, t}$,  for $1\leq t \leq s_c$. 	
\end{definition} 

We have the following two results.
\begin{corollary}\label{chaincoroA} 
Let $\calF_1, \calF_2$ be compactly supported $\omega$-semistable sheaves on $Y$ of slope $\mu$ which belong to two different chains respectively. 
Then one has: 
\begin{align}
	\Ext^p_Y(\calF_1, \calF_2) =0=\Ext^p_Y(\calF_2, \calF_2)\ ,
\end{align}
for all $p\in \Z$. 
\end{corollary}

\begin{corollary}\label{chaincorB} 
Let $\calF$ be an $\omega$-semistable compactly supported sheaf on $Y$ of slope $\mu$. For each $1\leq c \leq C_\mu$, let $\calF_c$ be the maximal 
$\omega$-semistable subsheaf of $\calF$ of slope $\mu(\calF_c)=\mu(\calF)$ which belongs to the chain $(\alpha_{c,1}, \ldots, \alpha_{c,s_c})$. Then there is an isomorphism 
\begin{align}\label{eq:ssdecomp}
	\calF\simeq \bigoplus_{c=1}^{C_\mu}\, \calF_c\ . 
\end{align}
\end{corollary} 

\begin{proof} 
The proof will proceed by induction on the length  $\ell_\calF$ of the Jordan-Hölder filtration of $\calF$. 

If $\ell_\calF=1$ the claim obviously holds since $\calF$ is $\omega$-stable. 

Suppose the claim holds for all sheaves $\calG$ satisfying the stated conditions, with $\ell_\calG \leq k$.  Let $\calF$ be an $\omega$-semistable sheaf with $\ell_\calF=k+1$. Then there exists a nonzero $\omega$-stable subsheaf $\calE\subset \calF$ with $\mu(\calE) = \mu(\calF)$ such that the quotient $\calG=\calF/\calE$ is nonzero. Then $\calG$ is also $\omega$-semistable with slope $\mu(\calG)=\mu(\calF)$ and $\ell_\calG =k$. By the induction hypothesis, 
\begin{align}
	\calG \simeq \bigoplus_{1\leq c\leq C_\mu}\, \calG_c\ .
\end{align}
Moreover, Corollary~\ref{chaincoroA} implies that all extension groups $\Ext^p_Y(\calG_c,\calE)$, with $p\geq 0$, as well as $\Ext_Y^p(\calE,\calG_c)$, with $p\geq 0$,  vanish for all $1\leq c \leq C_\mu$ such that $\calE$ does not belong to 
the associated chain ${\boldsymbol \alpha}_c$.  Then the snake lemma yields an exact sequence 
\begin{align}
	0 \longrightarrow \calE' \longrightarrow \calF \longrightarrow \calG' \longrightarrow 0\ ,
\end{align}
where 
\begin{align}
	\calG' = \bigoplus_{\genfrac{}{}{0pt}{}{1\leq c\leq C_\mu}{\calE\notin {\boldsymbol \alpha}_c}}\,  \calG_c \ ,
\end{align}
and $\calE'$ fits in an exact sequence 
\begin{align}
	0 \longrightarrow \calE \longrightarrow \calE' \longrightarrow \calG_{c_\calE} \longrightarrow 0  \ ,
\end{align}
where $c_\calE \in \{1, \ldots, C_\mu\}$ labels the unique chain $\calE$ belongs to. The claim then follows from Corollary~\ref{chaincoroA}. 
\end{proof} 

Finally, note that Lemma~\ref{finitewallAB} yields: 
\begin{corollary}\label{chainlemmB} 
	Let $({\alpha_1}, \ldots, {\alpha_s})$ be a chain in ${\sf S}_\mu$ and let $\calF_{\alpha_t}$ be a representative of $\alpha_t$ for $1\leq t\leq s$. Then there exists a unique line bundle $\calN$, up to isomorphism, so that $\calF_{\alpha_t}\otimes \calN\simeq \scrO_{\Delta_t}$ for all $1\leq t \leq s$. 
\end{corollary} 
\begin{proof}
	For each $1\leq t \leq s$, let $\calN_t$ denote the line bundle corresponding to $\calF_{\alpha_t}$ via Lemma~\ref {finitewallAB}. Using the defining properties of chains, the orthogonality relations \eqref{eq:orthrel} imply:
	\begin{align}
	\calN_t\otimes \scrO_{\Delta_u} \simeq \begin{cases}
		\calN_t\otimes \scrO_{\Delta_t} & \text{for } t=u \\[2pt]
		\scrO_{\Delta_u} & \text{for } t\neq u
	\end{cases} \ ,
	\end{align}
	for any $1\leq t,u \leq s$. Set $\calN \coloneqq \otimes_{t=1}^\ell \calN_t$.	
\end{proof}

\subsubsection{Exceptional collections and finite Dynkin quivers}\label{chainquivsect} 

Let ${\overline Y}$ be the natural toric completion of $Y$, which is a weighted projective space with two quotient singularities at infinity. Let $S$ be the minimal toric resolution of singularities of ${\overline Y}$ obtained by triangulating the toric polytope. Then  $S$ is a smooth projective surface containing $Y$ as an open subset. Moreover for each divisor $D_i$ there is a unique compact effective divisor ${\overline D}_i$ on $S$ such that the complement ${\overline D}_i \setminus D_i$ is zero dimensional. Moreover in the intersection ring of $S$ the orthogonality relations
\begin{align}\label{eq:orthrelD} 
	C_i \cdot {\overline D}_{j} = \delta_{i,j}\ ,
\end{align}
hold for $1\leq i,j\leq N$.

Let $({\alpha_1}, \ldots, {\alpha_s})$ be a chain of stable sheaves of $\omega$-slope $\mu$ as in Definition~\ref{chaindef}.  Using the orthogonality relations \eqref{eq:orthrelD},  Corollary~\ref {chainlemmB} implies that there exists a line bundle $\calM$ on $S$ such that any  sheaf $\calF_t$ in the equivalence class $\alpha_t$, for $1\leq t \leq s$, is isomorphic to:
\begin{align}\label{eq:Mtwist} 
	\calF_t = \begin{cases}
	\calM \otimes \scrO_{\Delta_t}(-\Delta_{t+1}) &  \text{for }  1\leq t \leq s-1\ ,\\[4pt] 
	\calM \otimes \scrO_{\Delta_s} & \text{for } t =s\ .	
	\end{cases}
\end{align}
For any $0\leq t \leq s-1$, let 
\begin{align}
	\calL_t \coloneqq \calM\otimes \scrO_S\Big(-\sum_{u=t+1}^{s}\Delta_u\Big)\ . 
\end{align}
Set also $\calL_s\coloneqq \calM$ and $\calF_0 \coloneqq \calL_0$. Then one has: 
\begin{lemma}\label{chainlemmA} 
	\hfill
	\begin{enumerate}[label=(\roman*)]\itemsep=0.2cm
		\item The sequence $(\calL_0,  \ldots, \calL_s)$ is an exceptional sequence of line bundles on $S$, i.e., it satisfies 
		\begin{align}\label{eq:exceptseqA} 
			\Ext_S^0(\calL_t, \calL_t) \simeq \C\ , \quad 
			\Ext_S^p(\calL_t, \calL_t) = 0\text{ for } p \geq 1 \ , 
		\end{align}
		and
		\begin{align}\label{eq:exceptseqB} 
			\Ext^p(\calL_t, \calL_u) =0\text{ for } p\geq 0 \ ,
		\end{align}
		for all $0\leq u<t\leq s$. In addition,
		\begin{align}\label{eq:exceptseqC} 
			\Ext_S^p(\calL_t,\calL_u) \simeq \C\text{ for } 0\leq p \leq 1\ ,
		\end{align} 
		for any $0\leq t < u \leq s$, and 
		\begin{align}\label{eq:exceptseqD} 
			\Ext_S^p(\calL_t, \calL_u) = 0\text{ for } p \geq 2\ ,
		\end{align}
		for all $0\leq t,u \leq s$. 
		
		\item Moreover, 
		\begin{align}\label{eq:exceptseqE} 
			\Ext^0_S(\calL_t, \calF_u) \simeq \begin{cases}
			\C & \text{if } t =u\ , \\
			0 & \text{ otherwise}\ ,
			\end{cases}
			\quad\text{and}\quad
			\Ext^1_S(\calL_t, \calF_u) \simeq \begin{cases}
			\C & \text{if }  u = t+1\ ,\\
			0 & \text{otherwise}\ ,
			\end{cases} 
		\end{align}
		as well as 
		\begin{align}\label{eq:exceptseqF} 
			\Ext^2_S(\calL_t, \calF_u) = 0
		\end{align} 
		for all $0\leq t,u\leq \ell$.
	\end{enumerate}
\end{lemma} 
\begin{proof}
	Claim (i) follows from Corollary~\ref{cor:cohcorA}, while Claim (ii) follows from Lemma~\ref{lem:cohcompA}.
\end{proof}

Given an exceptional sequence of line bundles as in Lemma~\ref {chainlemmA}, let $\calT\coloneqq \langle \calL_0, \ldots, \calL_s\rangle$ denote the smallest triangulated full subcategory of $\catDb(\catCoh(S))$ containing all  $\calL_t$ for $0\leq t \leq s$. Let also $\calC$ denote its intersection with $\catCoh(S)$. 
Note that \cite[Theorem~2.5]{Tilting_chains} applies to the present situation although the divisors $\Delta_t$ are reducible nodal curves, as opposed 
to smooth rational curves as assumed in \textit{loc.cit.}; more precisely, Lemma~\ref {chainlemmA} proves that the objects  $\calL_t$ and $\calF_t$ for $0\leq t \leq u$ satisfy all necessary conditions stated in  \cite[\S1.3]{Tilting_chains}. Then, Theorem 2.5 of \textit{loc.cit.} proves that the triangulated category $\CT$ is equivalent to the derived category of modules of a certain associative algebra determined by the line bundles $\calL_0, \ldots, \calL_s$. Moreover this equivalence identifies $\calC$ with the abelian category  of modules over the same algebra. This construction is further studied in more detail \cite[Propositions~2.8 and 6.18]{Noncomm_chains}. The results needed in the present context are summarized below. 

\begin{proposition}[{\cite[Theorem~2.5]{Tilting_chains}, \cite[Proposition~2.8]{Noncomm_chains}}]\label{chainpropA} 
	\hfill
	\begin{enumerate}[label=(\roman*)]\itemsep0.2cm
		
		\item Up to isomorphism, the simple objects of $\calC$ are $\sigma_0\coloneqq \calL_0$ and $\sigma_t \coloneqq \calF_t$ for $1\leq t \leq s$. 
		
		\item Up to isomorphism, there are $\ell+1$ unique indecomposable projective objects $\Lambda_0, \ldots, \Lambda_s$ in $\calC$, which fit in 
		nontrivial extensions of the form 
		\begin{align}\label{eq:projobA} 
			0\longrightarrow \Lambda_{t} \longrightarrow \Lambda_{t-1}\longrightarrow \calL_{t-1} \longrightarrow 0\ , 
		\end{align}
		for $1\leq t\leq s $, where $\Lambda_s = \calL_s$. In particular, at each step $\Ext^1_S(\calL_{t-1}, \Lambda_t)\simeq \C$, hence the extension \eqref{eq:projobA} is unique up to isomorphism.  
		
		\item \label{item:chainpropA-(iii)} The direct sum
		\begin{align}
		\Lambda \coloneqq \bigoplus_{t=0}^s\, \Lambda_s
		\end{align}
		is a projective generator of $\calC$. The tilting functor 
		$\R\Hom_{\catDb(\catCoh(S))}(\Lambda, -)$ yields equivalence of triangulated categories 
		\begin{align}\label{eq:tiltchainB} 
			\R\Hom_{\catDb(\catCoh(S))}(\Lambda, -)\colon \calT \simeq \catDb(\catMod(\End(\Lambda)))
		\end{align}
		identifying $\calC$ with the heart of the natural t-structure on the right-hand-side. 
		
		\item The following orthogonality relations 
		\begin{align}\label{eq:orthrelF} 
		\R\Hom_{\catDb(\catCoh(S))} (\Lambda_t, \sigma_u) = {\underline \C}\delta_{t, u}
		\end{align}
		hold in $\catDb(\catCoh(S))$. 
		
		\item \label{chainpropA:equivalence_abelian} Let $\calC_{\leq 1}$ be the abelian subcategory of $\calC$ consisting of coherent sheaves with at most one-dimensional support. Then $\calC_{\leq 1}$ is equivalent, via the natural inclusion $\calC_{\leq 1} \to \catCoh(S)$, to the subcategory of $\catCoh(S)$ consisting of coherent sheaves which admit a filtration with successive quotients $\{\sigma_1, \ldots, \sigma_s\}$. 
		
	\end{enumerate}  
\end{proposition}

Moreover, the following result follows immediately from \cite[Proposition~6.18]{Noncomm_chains}.
\begin{proposition}\label{chainpropB}
	There is an isomorphism of associative algebras $\End(\Lambda) \simeq \Lambda_{s, s-1}$, where $\Lambda_{s, s-1}$ is the Auslander algebra of $\C[x]/(x^s)$.
\end{proposition}

Lemma~\ref {chainlemmA}, Propositions~\ref {chainpropA}-\ref{chainpropA:equivalence_abelian} and~\ref {chainpropB} yield: 
\begin{corollary}\label{chaincoroC}
	There is an equivalence of abelian categories $\calC_{\leq 1} \simeq \mathsf{mod}(\Pi_{A_s})$ mapping the simple objects $\sigma_t$, for $1\leq t \leq s$, to the standard simple objects associated to the vertices of the quiver. Here, $\Pi_{A_s}$ is the preprojective algebra of the finite Dynkin quiver of type $A_s$.
\end{corollary}

\section{McKay correspondence for categorified Hall algebras}\label{s:Cat-McKay}

In this section, we shall use the results in \S\ref{s:semistable} and~\ref {s:chain} to establish a version of the McKay correspondence at the level of \textit{2-Segal spaces} and \textit{categorified Hall algebras}. We have reduced to the minimum the explicit description of the machinery from derived algebraic geometry in favor of a less technical description of the results. We will refer to \cite[\S4]{Porta_Sala_Hall} for an introduction of 2-Segal spaces and categorified Hall algebras.

\subsection{McKay correspondence and categorified Hall algebras}\label{ss:Cat-McKay}

In this section, we shall denote by $\catQCoh(Y)$ and by $\Pi_\calQ\Mod$ the stable $\infty$-categories\footnote{If he wishes, the reader can safely think of them as dg-categories.} of quasi-coherent complexes of $Y$ and of $\Pi_\calQ$-modules, respectively.
Since $Y$ is smooth and of finite type, $\catQCoh(Y)$ is of finite type in the sense of e.g. \cite[Definition~2.4-(7)]{Toen_Vaquie_Moduli}.\footnote{This follows formally from \cite[Theorem 1.4]{Efimov_Finiteness}, although in this case a much simpler argument is possible.}
On the other hand, \cite[Theorem~1.1]{Keller_Deformed_Erratum} implies that $\Pi_\calQ\Mod$ is of finite type as well.\footnote{To be precise, $\Pi_\calQ\Mod$ is the dg-category of right-modules over Ginzburg dg-algebra of $\calQ$. It corresponds to Keller's 2-Calabi-Yau completion of the free dg-category associated to the quiver $\calQ$ --- see  \cite[\S5.2]{BCS_Critical_loci} for details.}

It is shown in \cite[Lemma~3.2.8]{VdB_Flops} that the perverse coherent sheaf $\calP$ of \eqref{eq:projgenZN} satisfies
\begin{align}
	\R\Hom_{\catQCoh(Y)}(\calP,\calP) \simeq \Hom_{\catP(Y/X)}(\calP,\calP) \simeq \Pi_\calQ  \ . 
\end{align}
It follows that $\calP$ induces an $\infty$-functor
\begin{align}\label{eq:inftyMcKay}
	\R\Hom_{\catQCoh(Y)}(\calP, -) \colon \catQCoh(Y)\xrightarrow{\sim} \Pi_\calQ\Mod\ .
\end{align}
Combining the smoothness of $Y$ with \textit{loc. cit.}, we deduce that this is an equivalence, refining \eqref{eq:dercateqB} at the level of dg-categories.

By passing to moduli stacks, we have an equivalence $\calM_{\catQCoh(Y)}\simeq \calM_{ \Pi_\calQ\Mod}$ at the level of the corresponding moduli stacks of Toën-Vaquié's pseudo-perfect objects.\footnote{See \cite{Toen_Vaquie_Moduli} for the construction of these derived moduli stacks.} This lifts to an equivalence between the corresponding 2-Segal spaces:
\begin{align}\label{eq:equivalenceperfect-2Segal}
	\calS_\bullet \calM_{\catQCoh(Y)}\simeq \calS_\bullet \calM_{ \Pi_\calQ\Mod}\ .
\end{align}

Via the equivalence \eqref{eq:inftyMcKay}, we define a \textit{perverse $t$-structure} $\tau$ on $\catQCoh(Y)$ corresponding to the standard $t$-structure of $\Pi_\calQ\Mod$ (with such a choice, the equivalence \eqref{eq:inftyMcKay} is exact). Following \cite[\S2]{DPS_Torsion_Pairs}, we denote by $\bfPerCoh_{\mathsf{ps}}(Y/X)$ the derived moduli stack of properly supported $\tau$-coherent objects in $\catQCoh(Y)$ and we denote by $\bfCoh_{\mathsf{ps}}(\Pi_\calQ)$ the derived moduli stack of properly supported coherent objects in $\Pi_\calQ\Mod$ (here, the coherence is with respect to the standard $t$-structure). As proved in \textit{loc.cit.}, both moduli stacks are geometric derived stacks locally almost of finite presentation over $\C$. Moreover, the equivalence \eqref{eq:cateqB} yields
\begin{align}
	\bfPerCoh_{\mathsf{ps}}(Y/X)\simeq \bfCoh_{\mathsf{ps}}(\Pi_\calQ)\ .
\end{align}

The construction of categorified Hall algebras for dg-categories of finite type with nice enough $t$-structures done in \cite{DPS_Torsion_Pairs} yields:
\begin{proposition}\label{prop:CHA}
	There exists an equivalence of 2-Segal objects
	\begin{align}\label{eq:equivalencemoduliobjects-2Segal}
		\calS_\bullet\bfPerCoh_{\mathsf{ps}}(Y/X) \simeq \calS_\bullet \bfCoh_{\mathsf{ps}}(\Pi_\calQ) \ .
	\end{align}
	It induces $\mathbb E_1$-monoidal dg-category structures on
	\begin{align}
		\catCohb(\bfPerCoh_{\mathsf{ps}}(Y/X))\quad \text{ and }\quad \catCohb(\bfCoh_{\mathsf{ps}}(\Pi_\calQ))
	\end{align}
	for which there is an equivalence of $\mathbb E_1$-monoidal $\infty$-categories
	\begin{align}
		\catCohb(\bfPerCoh_{\mathsf{ps}}(Y/X)) \simeq  \catCohb(\bfCoh_{\mathsf{ps}}(\Pi_\calQ)) \ .
	\end{align}
\end{proposition}

At the level of $\sfG_0$-theory and Borel-Moore homology, the above proposition implies:

\begin{corollary}\label{cor:COHA}
	There exist associative algebra structures on 
	\begin{align}
		\sfG_0(\bfPerCoh_{\mathsf{ps}}(Y/X))\quad \text{ and }\quad \sfG_0(\bfCoh_{\mathsf{ps}}(\Pi_\calQ))\ ,
	\end{align}
	for which there is an isomorphism of associative algebras
	\begin{align}
		\sfG_0(\bfPerCoh_{\mathsf{ps}}(Y/X)) \simeq  \sfG_0(\bfCoh_{\mathsf{ps}}(\Pi_\calQ)) \ .
	\end{align}
	These results hold also for the Borel-Moore homology, and equivariantly. 
	
	Let $\zeta\in \Q^N$ be a stability condition. Then the above results hold for the open substack $\bfPerCoh_{0}^{\zeta\textrm{-}\mathsf{ss}}(Y/X)$ (resp.\ $\bfCoh_{0}^{\zeta\textrm{-}\mathsf{ss}}(\Pi_\calQ)$) of $\bfPerCoh_{\mathsf{ps}}(Y/X)$ (resp.\ of $\bfCoh_{\mathsf{ps}}(\Pi_\calQ)$) parameterizing $\zeta$-semistable perverse coherent sheaves on $Y$ of zero slope (resp.\ $\zeta$-semistable representations of $\Pi_\calQ$ of zero slope).
\end{corollary}

\subsubsection{The semistable McKay correspondence}

Let $\omega$ be a $\Q$-polarization of $Y$ and $\mu\in \Q_{>0}\cup\{\infty\}$. Set
\begin{align}
	\zeta_i\coloneqq \langle \omega, C_i\rangle
\end{align}
for $1\leq i \leq r$, and
\begin{align}
	\zeta_{r+1}\coloneqq\begin{cases}
		\displaystyle	\frac{1}{\mu}  - \sum_{i=1}^{r}m_i\, \zeta_i\ & \text{if } \mu\neq \infty\ ,\\[8pt]
		\displaystyle	- \sum_{i=1}^{r} m_i\, \zeta_i &\text{otherwise}\ .
	\end{cases}
\end{align}

We denote by $\bfCoh_\mu^{\omega\textrm{-}\mathsf{ss}}(Y)$ the derived stack of $\omega$-semistable properly supported coherent sheaves on $Y$ of slope $\mu$. As proved in  \cite{Porta_Sala_Hall}, there exists a 2-Segal space $\calS_\bullet \bfCoh_\mu^{\omega\textrm{-}\mathsf{ss}}(Y)$ and an $\mathbb E_1$-monoidal dg-category structure on $\catCohb( \bfCoh_\mu^{\omega\textrm{-}\mathsf{ss}}(Y) )$. Theorem~\ref{categmckay} lifts to an equivalence at the level of derived moduli stacks:
\begin{align}
	\bfCoh_\mu^{\omega\textrm{-}\mathsf{ss}}(Y) \simeq \bfPerCoh_{0}^{\zeta\textrm{-}\mathsf{ss}}(Y/X) \simeq \bfCoh^{\zeta\textrm{-}\mathsf{ss}}_0(\Pi_\calQ)\ .
\end{align}
Since the 2-Segal spaces structures on $\calM_{\catQCoh(Y)}$ and $\calM_{\Pi_\calQ\Mod}$ are compatible with the McKay equivalence (cf.\ Formula~\eqref{eq:equivalenceperfect-2Segal}), the above equivalences upgrade to equivalences at the level of 2-Segal spaces:
\begin{align}
	\calS_\bullet \bfCoh_\mu^{\omega\textrm{-}\mathsf{ss}}(Y) \simeq \calS_\bullet \bfPerCoh_{0}^{\zeta\textrm{-}\mathsf{ss}}(Y/X) \simeq \calS_\bullet 	\bfCoh^{\zeta\textrm{-}\mathsf{ss}}_0(\Pi_\calQ)\ .
\end{align}
Therefore, we obtain our version of the McKay correspondence for categorified Hall algebras.
\begin{theorem}\label{thm:McKay-ss}
	There exists an equivalence
	\begin{align}
			\catCohb( \bfCoh_\mu^{\omega\textrm{-}\mathsf{ss}}(Y) )\simeq \catCohb(\bfCoh^{\zeta\textrm{-}\mathsf{ss}}_0(\Pi_\calQ) )
	\end{align}
	as $\mathbb E_1$-monoidal dg-categories. Moreover, the same holds equivariantly with respect to a diagonal torus $T\subset \GL(2,\C)$ centralizing the finite group $G$.
\end{theorem}
\begin{corollary}
	There are isomorphisms of associative algebras
	\begin{align}
		\sfG_0( \bfCoh_\mu^{\omega\textrm{-}\mathsf{ss}}(Y)  ) \simeq \sfG_0( \bfCoh^{\zeta\textrm{-}\mathsf{ss}}_0(\Pi_\calQ) ) \quad\text{and}\quad \sfH_\ast^{\mathsf{BM}}( \bfCoh_\mu^{\omega\textrm{-}\mathsf{ss}}(Y)  ) \simeq \sfH_\ast^{\mathsf{BM}}( \bfCoh^{\zeta\textrm{-}\mathsf{ss}}_0(\Pi_\calQ) )\ .
	\end{align}
	Moreover, the same holds equivariantly with respect to a diagonal torus $T\subset \GL(2,\C)$ centralizing the finite group $G$.
\end{corollary}

\subsection{Categorified McKay correspondence in type A}

Let $G$ be $\Z_{N+1}$ for $N\in \Z$, $N\geq 1$ (hence, $r=N$). From now on, let $\mu\neq \infty$. As explained in Remark~\ref{partitionchains}, there is a unique partition into pairwise disjoint subsets 
\begin{align}
	\sfS_\mu = \bigcup_{c=1}^{C_\mu} \{\, \alpha_{c,1}, \ldots, \alpha_{c,s_c}\, \}
\end{align}
such that for each $1\leq c\leq C_\mu$ the ordered sequence 
\begin{align}	
	\balpha_c \coloneqq (\alpha_{c,1}, \ldots, \alpha_{c, s_c})
\end{align}
is a chain for all $1\leq c \leq C_\mu$. 

Denote by $\bfCoh^{\omega\textrm{-}\mathsf{ss}}_{\balpha_c}(Y)$ the moduli stack of $\omega$-semistable properly supported sheaves on $Y$ which belongs $\balpha_c$ (cf.\ Definition~\ref{JHchain}). Corollaries~\ref{chaincoroA} and \ref {chaincorB} implies that the partition of $\sfS_\mu$ into chains determines a natural isomorphism of stacks
\begin{align}\label{eq:prodstackisomB}
	\bfCoh_\mu^{\omega\textrm{-}\mathsf{ss}}(Y) \simeq \bigtimes_{c=1}^{C_\mu}\, \bfCoh^{\omega\textrm{-}\mathsf{ss}}_{\balpha_c}(Y)\ .
\end{align}
On the other hand, by Proposition~\ref{chainpropA}-\ref{item:chainpropA-(iii)}, Proposition~\ref {chainpropB} and Corollary~\ref {chaincoroC} we have 
\begin{align}
	\bfCoh^{\omega\textrm{-}\mathsf{ss}}_{\balpha_c}(Y)\simeq \bfCoh_{\mathsf{ps}}(\Pi_{A_{s_c}})\ ,
\end{align}
for each $c=1, \ldots, C_\mu$. Here $A_{s}$ denotes the type A Dynkin diagram with $s$ vertices. We can upgrade the above isomorphism to an isomorphism
\begin{align}
		\calS_\bullet \bfCoh_\mu^{\omega\textrm{-}\mathsf{ss}}(Y) \simeq \bigtimes_{c=1}^{C_\mu}\, \calS_\bullet  \bfCoh_{\mathsf{ps}}(\Pi_{A_{s_c}})
\end{align}
of 2-Segal spaces. Thus, we obtain the main result of this section: 
\begin{theorem}\label{thm:McKay-chain} 
	Let $\omega$ be a $\Q$-polarization of $Y$ and $\mu\in \Q_{>0}$. Let $\lambda_\mu\coloneqq (s_1, \ldots, s_{C_\mu})$ be the partition associated with $\mu$ as in Remark~\ref {partitionchains}. Then we have a functor
	\begin{align}
		 \bigotimes_{c=1}^{C_\mu}\, \catCohb( \bfCoh_{\mathsf{ps}}(\Pi_{A_{s_c}}) ) \longrightarrow	\catCohb( \bfCoh_\mu^{\omega\textrm{-}\mathsf{ss}}(Y) ) 
	\end{align}
	of $\mathbb E_1$-monoidal dg-categories. It induces isomorphisms of associative algebras
	\begin{align}
		\sfG_0( \bfCoh_\mu^{\omega\textrm{-}\mathsf{ss}}(Y)  ) \simeq \bigotimes_{c=1}^{C_\mu}\,  \sfG_0( \bfCoh_{\mathsf{ps}}(\Pi_{A_{s_c}}) ) \quad\text{and}\quad \sfH_\ast^{\mathsf{BM}}( \bfCoh_\mu^{\omega\textrm{-}\mathsf{ss}}(Y)  ) \simeq \bigotimes_{c=1}^{C_\mu}\, \sfH_\ast^{\mathsf{BM}}( \bfCoh_{\mathsf{ps}}(\Pi_{A_{s_c}}) )\ .
	\end{align}
	Moreover, these results holds also equivariantly with respect to a diagonal torus $T\subset \GL(2,\C)$ centralizing $\Z_{N+1}$.
\end{theorem} 

\begin{remark}\label{rem:McKay-chain}
	Let $\omega$ be a $\Q$-polarization of $Y$ and $\mu\in \Q_{>0}$. Set $\zeta_i\coloneqq \langle \omega, C_i\rangle$ for $1\leq i \leq N$, and
	\begin{align}
		\zeta_{N+1}\coloneqq\frac{1}{\mu}  - \sum_{i=1}^{N}\, \zeta_i\ .
	\end{align}
	Then, Theorems~\ref{thm:McKay-ss} and \ref{thm:McKay-chain} imply that one has a functor
	\begin{align}
		\bigotimes_{c=1}^{C_\mu}\, \catCohb( \bfCoh_{\mathsf{ps}}(\Pi_{A_{s_c}}) ) \longrightarrow	\catCohb( \bfCoh^{\zeta\textrm{-}\mathsf{ss}}_0(\Pi_\calQ) ) 
	\end{align}
	of $\mathbb E_1$-monoidal dg-categories. It induces isomorphisms of associative algebras
	\begin{align}
		\sfG_0( \bfCoh^{\zeta\textrm{-}\mathsf{ss}}_0(\Pi_\calQ)  ) \simeq \bigotimes_{c=1}^{C_\mu}\,  \sfG_0( \bfCoh_{\mathsf{ps}}(\Pi_{A_{s_c}}) ) \quad\text{and}\quad \sfH_\ast^{\mathsf{BM}}( \bfCoh^{\zeta\textrm{-}\mathsf{ss}}_0(\Pi_\calQ) ) \simeq \bigotimes_{c=1}^{C_\mu}\, \sfH_\ast^{\mathsf{BM}}( \bfCoh_{\mathsf{ps}}(\Pi_{A_{s_c}}) )\ .
	\end{align}
	Moreover, these results holds also equivariantly.
\end{remark}

\subsubsection{Betti numbers and Kac polynomials}\label{ss:betti_numbers}

Let $\Pi_\calQ$ be the preprojective algebra of an arbitrary quiver $\calQ$, $\zeta\in\Z^{\calQ_0}$ a stability condition, and $\vartheta\in \Q$ a fixed slope. Then, by \cite[Theorems~6.4 and D]{Davison_integrality}, one gets 
\begin{align}\label{eq:Kac-polynomial}
	\sum_{\genfrac{}{}{0pt}{}{\bfd}{\bfd=0\text{ or }\mu(\bfd)=\vartheta}}\, \sum_{i\in\Z}\, \dim\, \sfH_i^{\mathsf{BM}}(\bfCoh(\Pi_\calQ; \bfd)^{\zeta\textrm{-}\mathsf{ss}}_\vartheta) q^{\langle\bfd,\bfd\rangle+i/2} y^{\bfd} = \mathsf{Exp} \Big(\frac{1}{q-1}\,\sum_{\bfd\neq 0}\, a_{\calQ,\, \bfd}^{\zeta\textrm{-}\mathsf{ss}}(q^{-1/2}) y^\bfd\Big)
\end{align}
 where $a_{\calQ,\, \bfd}^{\zeta\textrm{-}\mathsf{ss}}$ is the weight series of the \textit{BPS sheaf} on an associated coarse moduli space\footnote{The latter is a bounded complex of constructible sheaves on the coarse moduli space of representations of a quiver with a potential $({\widetilde \calQ}, W)$, canonically associated to $\Pi_\calQ$. The construction is fairly involved and we refer the reader to \cite{Davison_DT, Davison_integrality} for more details.} and $y^\bfd \coloneqq \prod_i y_i^{d_i}$ for any dimension vector $\bfd$.  

As before, fix $\calQ=A_N^{(1)}$ and a stability condition $\zeta\in \Z^{N+1}$ subjects to the following conditions: $\zeta_i>0$ for $i=1, \ldots, N$ and
\begin{align}
	\mu\coloneqq \frac{1}{\zeta_{N+1}+\sum_{i=1}^{N}\, \zeta_i}>0\ .
\end{align}
Following Remark~\ref{partitionchains}, let $\Delta_c^{+}$ denote the set of positive roots of the associated Lie algebra of type $A_{s_c}$ for any $1\leq c \leq C_\mu$. Let $\Delta_\mu^+$ be the disjoint union 
\begin{align}
	\Delta^+_\mu\coloneqq \bigsqcup_{c=1}^{C_\mu} \Delta_c^+\ .
\end{align}
Since, for each $1\leq c \leq C_\mu$, the simple roots of $A_{s_c}$ are identified by construction to the isomorphism classes $\alpha_{c,1}, \ldots, \alpha_{c, s_c}$, there is a natural map $\bfd \colon  \Delta^+_\mu \to \Z^{N+1}$ given by 
\begin{align}
	\bfd(\lambda_c)  = \sum_{\ell=i}^j\, \bfd(\alpha_{c, \ell})\ ,
\end{align}
if $\lambda_c = \sum_{\ell=i}^j\, \alpha_{c, \ell}\in \Delta_c^+$, with $1\leq c\leq C_\mu$. Here $\bfd(\alpha_{c,\ell})$ denote the associated dimension vector of an arbitrary representative $\calF_{\alpha_{c, \ell}}$ of $\alpha_{c,\ell}$ (cf.\ Remark~\ref{rem:stable-dimension-vector}). 

Recall that the Kac polynomial $a_{A_s,\, \bfe}(t)$ of the quiver $A_s$ and dimension vector $\bfe$ is nonzero if and only if $\bfe$ is a positive root, and in this case its value is exactly one. Then, by using Remark~\ref{rem:McKay-chain} and applying Formula~\eqref{eq:Kac-polynomial} twice, i.e., to $\calQ$ and $\zeta$ above and to the quivers $A_{s_c}$'s and the degenerate stability condition, yield to the following result.
\begin{proposition}
	We have:
	\begin{align}
		\sum_{\bfd\in \Z^{N+1}, \, \bfd\neq 0}\, a_{\calQ,\, \bfd}^{\zeta\textrm{-}\mathsf{ss}}(q^{-1/2}) y^\bfd
		= \sum_{c=1}^{C_\mu}\, \sum_{\lambda_c\in \Delta^+_c}\, y^{\bfd(\lambda_c)}\ .
	\end{align} 
\end{proposition}

\appendix

\section{Perverse coherent sheaves and tilting by a torsion pair}\label{s:yoshioka}

The goal of this section is to formulate and prove a version of \cite[Proposition~2.1.1]{YoshiokaPerverse}, which shows that $\catP(Y/X)$ can be obtained as a tilting by a torsion pair, in our \textit{local} setting.

\begin{definition}
	Let $\scrA$ be an abelian category. 
	A \textit{torsion pair} in $\scrA$ is a pair $\upsilon = (\scrT, \scrF)$ of full subcategories such that
	\begin{itemize}\itemsep0.2cm
		\item for any $T\in \scrT$ and $F\in \scrF$ one has $\Hom_{\scrA}(T, F)=0$; 
		\item every $X\in \scrA$ fits into an exact sequence
		\begin{align}
			0\longrightarrow T\longrightarrow E\longrightarrow F\longrightarrow 0
		\end{align}
		with $T\in \scrT$ and $F\in \scrF$.
	\end{itemize}
	In this case, we refer to $\scrT$ as the \textit{torsion part} of $\upsilon$ and to $\scrF$ as the \textit{torsion-free part} of $\upsilon$.
\end{definition}

The proof of next lemma consists of standard category theory and we leave it to the reader.
\begin{lemma}\label{lem:torsion_pairs_extensions_mono_epi}
	Let $\scrA$ be an abelian category and let $\upsilon = (\scrT, \scrF)$ be a torsion pair on $\scrA$.
	Then:
	\begin{enumerate}\itemsep=0.2cm
		\item both $\scrT$ and $\scrF$ are closed under extensions;
		
		\item if $T \to T'$ is an epimorphism in $\scrA$ and $T \in \scrT$, then $T' \in \scrT$;
		
		\item if $F' \to F$ is a monomorphism in $\scrA$ and $F \in \scrF$, then $F' \in \scrF$.
	\end{enumerate}
\end{lemma}

Let $\calP$ be the locally free sheaf introduced in \eqref{eq:projgenZN}. Define
\begin{align}
	T(\calP) &\coloneqq \{ \calF\in \catCoh(Y)\, \vert\, \Ext^1(\calP, \calF)=0 \}\ ,  \\[4pt]
	S(\calP) &\coloneqq \{ \calF\in \catCoh(Y)\, \vert\, \Hom(\calP, \calF)=0 \}\ .
\end{align}

\begin{remark}\label{rem:torsion-pair}
	Note that $\calP\in T(\calP)$. Moreover, any coherent sheaf $\calE\in S(\calP)$ must be supported on $C$. Indeed,  by \cite[Chapter~3, Proposition~8.5]{Hartshorne_AG},  $\Hom(\calP, \calF) =0$ implies $\pi_\ast(\calP^\vee \otimes \calF)=0$ since $X$ is  affine. Since $\pi$ is an isomorphism on the complement of $C_{\mathsf{red}}$, $\calP^\vee \otimes\calF$ must be supported on $C_{\mathsf{red}}$, hence the same for $\calF$ since $\calP$ is a locally free sheaf.
\end{remark}

\begin{proposition}[{cf.\  \cite[Lemma~1.1.11-(1) and Proposition~1.1.13]{YoshiokaPerverse}}]\label{prop:tilting}
	The pair $(T(\calP), S(\calP))$ is a torsion pair  of $\catCoh(Y)$, whose tilting is $\catP(Y/X)$.
\end{proposition}

\begin{proof}
	First, note that we shall use the fact that for any coherent sheaf $\calF$ on $Y$, one has $\R^i \pi_\ast (\calF)=0$ for any $i\geq 2$ by \cite[Theorem~18.8.5]{Vakil_AG}, since the fibers of $\pi$ have dimension at most one. Thanks to  \cite[Chapter~3, Proposition~8.5]{Hartshorne_AG}, this implies that $H^i(Y, \calF)=0$ for $i\geq 2$.
	
	We need to prove that for any $\calT \in T(\calP)$ and $\calF\in  S(\calP)$, one has $\Hom(\calT, \calF)=0$. Let us assume that there exists  $f\in \Hom(\calT, \calF)$, $f\neq 0$. By applying $\Hom(\calP, -)$ to the short exact sequences
	\begin{align}
		0\longrightarrow \ker(f)\longrightarrow \calT \longrightarrow \mathsf{Im}(f) \longrightarrow 0 \ ,\\[4pt]
		0\longrightarrow \mathsf{Im}(f)\longrightarrow \calF \longrightarrow \mathsf{Coker}(f) \longrightarrow 0 \ ,
	\end{align}
	and using the definitions of $T(\calP)$ and $S(\calP)$, we get $\Ext^1(\calP, \mathsf{Im}(f))=0=\Hom(\calP, \mathsf{Im}(f))$. Thus, $\R\Hom(\calP, \mathsf{Im}(f))=0$. Hence, by the equivalence \eqref{eq:dercateqB}, one gets $\mathsf{Im}(f)=0$. 
	
	Let $\calE$ be a coherent sheaf on $Y$. Consider now the evaluation map
	\begin{align}
		\ev \colon \pi^\ast \pi_\ast (\calP^\vee \otimes \calE)\otimes \calP \longrightarrow \calE \ .
	\end{align}
	Set $\calT\coloneqq \mathsf{Im}(\ev)$ and $\calF\coloneqq \mathsf{Coker}(\ev)$. Consider the short exact sequence
	\begin{align}
		0\longrightarrow \calT \longrightarrow \calE \longrightarrow \calF\longrightarrow 0\ .
	\end{align}
	We want to prove that $\calT\in T(\calP)$ and $\calF\in S(\calP)$. First, note that the above sequence yields the long exact sequence
	\begin{align}
		0\longrightarrow \pi_\ast(\calP^\vee\otimes \calF) \longrightarrow \R^1\pi_\ast(\calP^\vee \otimes \calT) \longrightarrow \cdots 
	\end{align}
	Therefore it suffices to prove that 
	\begin{align}
		\R^1\pi_\ast(\calP^\vee \otimes \calT) =0\ . 
	\end{align}
	The exact sequence
	\begin{align}
		0 \longrightarrow \ker(\ev)\otimes \calP^\vee \longrightarrow  \pi^\ast \pi_\ast (\calP^\vee \otimes \calE)\otimes \calP\otimes \calP^\vee \longrightarrow \calT\otimes \calP^\vee \longrightarrow 0
	\end{align}
	yields a surjection
	\begin{align}
		\cdots \longrightarrow  \R^1\pi_\ast(\pi^\ast \pi_\ast (\calP^\vee \otimes \calE)\otimes \calP\otimes \calP^\vee) \longrightarrow \R^1\pi_\ast (\calT\otimes \calP^\vee) \longrightarrow 0 \ ,
	\end{align}
	since $\R^2\pi_\ast (\ker(\ev)\otimes \calP^\vee)=0$ by the arguments at the beginning of the proof. Hence it suffices to prove that
	\begin{align}
		 \R^1\pi_\ast(\pi^\ast \calG\otimes \calP\otimes \calP^\vee) =0 \ ,
	\end{align}
	where $\calG \coloneqq \pi_\ast(\calP^\vee\otimes \calE)$. Since $X$ is quasi-projective, there is a coherent locally free $\scrO_X$-module $\calV$ and a surjection 
	$\calV \xrightarrow{v} \calG$. Then $\pi^\ast v$ is also surjective, hence 
	we obtain an exact sequence 
	\begin{align}
		0\longrightarrow \calK \longrightarrow \pi^\ast\calV \longrightarrow \pi^\ast\calG \longrightarrow 0\ ,
	\end{align}
	where $\calK \coloneqq \ker(\pi^\ast v)$ is a coherent $\scrO_Y$-module. 
	This yields the exact sequence 
	\begin{align}
		\cdots \to \R^1\pi_\ast(\pi^\ast\calV \otimes \calP\otimes \calP^\vee) \longrightarrow 
	\R^1\pi_\ast(\pi^\ast\calG \otimes \calP\otimes \calP^\vee)\longrightarrow \R^2\pi_\ast(\calK\otimes 
	\calP\otimes \calP^\vee) \longrightarrow \cdots 
	\end{align}	
	Again, as explained at the beginning of the proof, we must have $\R^2\pi_\ast(\calK\otimes 
	\calP\otimes \calP^\vee) =0$. 
	Moreover, the projection formula yields an isomorphism 
\begin{align}
	\R^1\pi_\ast(\pi^\ast \calV \otimes \calP\otimes \calP^\vee) \simeq \calV \otimes
	\R^1\pi_\ast( \calP\otimes \calP^\vee) \ ,
\end{align}
	where the right hand side is identically zero since $\calP$ belongs to 
	$T(\calP)$. In conclusion,  
	\begin{align}
		\R^1\pi_\ast(\pi^\ast\calG \otimes \calP\otimes \calP^\vee)=0
	\end{align}
	as claimed above.
	
	Let us denote by $\scrC(\calP)$ the tilted category associated to $(T(\calP), S(\calP))$, i.e.,
	\begin{align}
		\scrC(\calP)=\{ E\in \catDb(\catCoh(Y))\, \vert\, \calH^{-1}(E)\in S(\calP)\ ,  \calH^0(E)\in T(\calP) \ , \text{ and } \calH^i(E)=0 \text{ for } i\neq -1, 0 \} \ .
	\end{align}
	
	Let $E\in \catP(Y/X)$. We have the spectral sequence
	\begin{align}
		E_2^{p, q}\coloneqq \Ext^p(\calP, \calH^q(E)) \Longrightarrow \Ext^{p+q}(\calP, E) \ .
	\end{align}
	First, notice that it degenerates since $E_2^{p,q}=0$ for $p\neq 0, 1$. Moreover, $\Ext^i(\calP, E)=0$ for $i\neq 0$, by \cite[Corollary~3.2.8]{VdB_Flops}. Therefore, $\Hom(\calP,  \calH^{-1}(E)) =0$ and $\Ext^1(\calP, \calH^0(E))=0$. Thus, $\calH^{-1}(E)\in S(\calP)$ and $\calH^0(E)\in T(\calP)$. Therefore, $E\in \scrC(\calP)$.
	
	Conversely, $S(\calP)[1], T(\calP)\subset \catP(Y/X)$ since they are both mapped into $\catMod(\Pi_\calQ)$ by the equivalence \eqref{eq:dercateqB}. Therefore, $\scrC(\calP)\subset \catP(Y/X)$.
\end{proof}

Now, we analyze the simple objects of $\catP(Y/X)$. First, by \cite[Proposition~3.5.7]{VdB_Flops} 
\begin{align}
	\calI_j\coloneqq \begin{cases}
		\scrO_{C} & \text{for } j=r+1\ ,\\
		\scrO_{C_j}(-1)[1] & \text{for } j=1, \ldots, r\ .
	\end{cases}
\end{align}
corresponds to the $j$-th simple module $\calS_j$ associated to the vertex $j$ of the affine quiver, for $j=1, \ldots, r+1$. Moreover, one has the following:
\begin{lemma}\label{lem:simple} 
	Let $p\in Y \setminus C_{\mathsf{red}}$. Then $\tau(\scrO_p)$ is a simple object, which is not isomorphic to $\calS_1, \ldots, \calS_{r+1}$.
\end{lemma} 

\begin{proof}
	Note that $\scrO_p$ belongs to the abelian category of properly supported perverse coherent sheaves $\catPps(Y/X)$. Therefore $\tau(\scrO_p)$ is a finite-dimensional representation of $\Pi_\calQ$. Since each direct summand of the local projective generator $\calP$, defined in \eqref{eq:projgenZN}, has rank $\mathsf{rk}(\calE_i)= m_i$, for $1\leq i \leq r+1$, the dimension vector of $\tau(\scrO_p)$ is $\bfm\coloneqq(m_1, \ldots, m_r)$. 
	
	We have to show that $\tau(\scrO_p)$ is simple. Let us assume that it is not. Then it must be nilpotent by \cite[Lemma~2.31-(2)]{Tilting_moduli}, therefore its composition factors belong to $\{\calS_1, \cdots, \calS_{r+1}\}$ up to isomorphism (cf.\ \cite[Definition~2.14]{Tilting_moduli}). In particular $\tau(\scrO_p)$ admits a sub-representation $\calS_i\subset \tau(\scrO_p)$ for some $1\leq i \leq r+1$. This implies that there is an injective morphism 
	\begin{align}
		\tau^{-1}(\calS_i) \longrightarrow \scrO_p 
	\end{align}
	in $\catPps(Y/X)$. However, as described in Remark~\ref{rem:simple}, all objects $\tau^{-1}(\calS_i)$ are set-theoretically supported in $C_{\mathsf{red}}$, hence this leads to a contradiction under the current assumptions. In conclusion, $\tau(\scrO_p)$ must be simple if $p \in Y \setminus C_{\mathsf{red}}$.
	
	Since $\tau(\scrO_p)$ is finite-dimensional, by \cite[Lemma~2.31-(1)]{Tilting_moduli}, $\tau(\scrO_p)$ is not isomorphic to $\calS_1$, $\ldots$, $\calS_{r+1}$.\footnote{This is evident since $p\in Y \setminus C_{\mathsf{red}}$, while $\calS_j$ corresponds, via $\tau$, to coherent sheaves set-theoretically supported on $C$, for $j=1, \ldots, r+1$.}
\end{proof}

\begin{rem}
	One has $\scrO_p, \scrO_C\in T(\calP)$, where $p\in Y \setminus C_{\mathsf{red}}$, and $\scrO_{C_j}(-1)\in S(\calP)$ thanks to Proposition~\ref{prop:tilting}. \hfill $\triangle$
\end{rem}

\begin{lemma}\label{lem:simpleobj} 
	Suppose $E \in \catP(Y/X)$ is a simple object and properly supported. Then $E$ is isomorphic to one of the 
	$\calI_i$, for $i=1, \ldots, r+1$ or to $\scrO_p$, for a point $p \in Y\setminus C_{\mathsf{red}}$. 
\end{lemma} 

\begin{proof}
	Suppose $E$ is simple and not isomorphic to any of the 
$\calI_i$, for $i=1, \ldots, r+1$. Since $\tau(E)$ is finite-dimensional, \cite[Lemma~2.31-(1)]{Tilting_moduli} shows that 
its dimension vector must be $\bfd(E)= (m_1, \ldots, m_{r+1})$, and Lemma~\ref{dimvectlemm} yields $\ch_1(E)=0$ and 
$\chi(E) =1$. 

We claim that $\calH^{-1}(E)=0$. Suppose this is not the case. Since 
$\calH^{-1}(E)[1]$ is a sub-object of $E$, and $E$ is simple, one must have 
$\calH^{-1}(E)[1]\simeq E$. Then $\ch_1(\calH^{-1}(E))=0$, which implies that $\calH^{-1}(E)$ is a zero-dimensional sheaf. This further implies that $\chi(E)<0$, in contradiction with 
$\chi(E)=1$. 

In conclusion, $E\simeq \calH^0(E)$ is a zero-dimensional sheaf with $\chi(E)=1$. 
Hence $E \simeq \scrO_p$ for some reduced point  $p\in Y$. Suppose $p \in C_i$ 
for some $1\leq i \leq r$. Then one has an exact sequence 
\begin{align}
	0\longrightarrow \scrO_{C_i}(-1)[1] \longrightarrow \scrO_p \longrightarrow \scrO_{C_i} \longrightarrow 0 
\end{align}
in ${\sf P}(Y/X)$, which contradicts simplicity. Therefore $p \in Y \setminus C_{\mathsf{red}}$. 
\end{proof}

Define 
\begin{align}
	\Sigma& \coloneqq \{\scrO_{C_j}(-1) \,\vert\, j=1, \ldots, r\}\ ,\\[4pt]
	T& \coloneqq \{\calE\in \catCoh(Y)\, \vert\, \Hom(\calE, \calF)=0 \text{ for any } \calF\in \Sigma\}\ , \\[4pt]
	S & \coloneqq \{\calE\in \catCoh(Y)\, \vert\, \calE \text{ is a successive extension of subsheaves of } \calF\in \Sigma\} \ . 
\end{align}

Since $\Sigma\subset S(\calP)$, it is straightforward to see that $S\subseteq S(\calP)$. Let us prove the converse.
\begin{lemma}[{cf.\ \cite[Lemma~1.1.23]{YoshiokaPerverse}}]\label{lem:S}
	One has $S(\calP)\subseteq S$.
\end{lemma}

\begin{proof}
	We shall prove that for any $\calE\in S(\calP)$, there exists a filtration
	\begin{align}
		0\subset \calF_1\subset \cdots\subset \calF_r=\calE
	\end{align}
	such that for any $k=1, \ldots, r$, there exists $j\in\{1, \ldots, r\}$ and an injective morphism $\calF_k/\calF_{k-1} \to \scrO_{C_j}(-1)$.
	
	Since $\calE\in S(\calP)$, $\calE[1]\in \catP(Y/X)$. Then, there exists a surjective morphism from $\calE[1]$ to a simple object $G$ of $\catP(Y/X)$. By Remark~\ref{rem:torsion-pair}, $\calE[1]$ is supported on $C$, hence it is properly supported. Then, also $G$ is properly supported. Therefore, by Lemma~\ref{lem:simpleobj}, $G$ is either isomorphic to one of the $\calI_i$, for $i=1, \ldots, r+1$ or to $\scrO_p$, for a point $p \in Y\setminus C_{\mathsf{red}}$.   
	
	We need to exclude $\scrO_p$ with $p\in Y \setminus C_{\mathsf{red}}$, since $\calE$ is set-theoretically supported on $C$. Then, $\calE[1]\to \calI_j$ for some $j\in \{1, \ldots, r+1\}$. We want to prove that $j\neq r+1$. Consider the kernel $F\coloneqq \ker(\calE[1]\to \calI_j)$ in $\catP(Y/X)$. We have an exact sequence
	\begin{align}
		0\longrightarrow \calH^{-1}(F) \longrightarrow \calE \longrightarrow \calH^{-1}(\calI_j) \longrightarrow \calH^0(F) \longrightarrow 0 \longrightarrow \calH^0(\calI_j) \longrightarrow 0 \ .
	\end{align}
	Therefore, $\calH^0(\calI_j)=0$, hence $j\neq r+1$ and therefore $\calI_j\simeq \scrO_{C_j}(-1)[1]$. Moreover, $\calH^{-1}(F)\in S(\calP)$. Set $\calE'\coloneqq \mathsf{Im}(\calE \longrightarrow \calH^{-1}(\calI_j)\simeq \scrO_{C_j}(-1))$. We have a short exact sequence
	\begin{align}
		0\longrightarrow \calH^{-1}(F)\longrightarrow\calE \longrightarrow \calE'\longrightarrow 0 \ ,
	\end{align}
	together with an injective map $\calE/\calH^{-1}(F)\simeq \calE'\to \scrO_{C_j}(-1)$. Since $\calH^{-1}(F)\in S(\calP)$, we can iterate this procedure by induction on the support of $\calE$. Thus, we get the assertion.
\end{proof}

The above lemma implies that $S(\calP)\cap T=0$. Moreover, $T(\calP)\subseteq T$ since $\Sigma\subset S(\calP)$. Let us show the converse.
\begin{lemma}\label{lem:T}
	One has $T\subseteq T(\calP)$. 
\end{lemma}

\begin{proof}
	Let $\calE\in T$, i.e., let $\calE$ be a coherent sheaf on $Y$ for which $\Hom(\calE, \scrO_{C_j}(-1))=0$ for $j=1, \ldots, r$. Consider the short exact sequence 
	\begin{align}
		0\longrightarrow \calT \longrightarrow \calE \longrightarrow \calF\longrightarrow 0 \ ,
	\end{align} 
	with $\calT\in T(\calP)$ and $\calF\in S(\calP)$. We get that $\Hom(\calF, \scrO_{C_j}(-1))=0$ for $j=1, \ldots, r$. Thus, $\calF\in S(\calP)\cap T$. Hence, $\calF=0$ and $\calE\in T(\calP)$.
\end{proof}

From the previous considerations together with Lemmas~\ref{lem:S} and \ref{lem:T}, it follows that $T=T(\calP)$ and $S=S(\calP)$. Therefore, we obtain the main result of this section.
\begin{proposition}[{cf.\ \cite[Propositions~1.1.26 and 2.1.1-(2)]{YoshiokaPerverse}}]\label{prop:yoshioka_prop_2.1.1}
	The pair $(T, S)$ is a torsion pair of $\catCoh(Y)$ whose tilting is $\catP(Y/X)$.
\end{proposition}

\section{Some sheaf cohomology group computations}

Let $S$ be a smooth rational projective surface and let $C_1, \ldots, C_k$, for $k\geq 2$, be an $A_k$-chain of smooth rational $(-2)$-curves on $S$. For any $1\leq i, j\leq k$ let
\begin{align}
	C_{i,j} \coloneqq \begin{cases}
		C_i+\cdots + C_j & \text{for }\ i\leq j\ , \\
		0 & \text{otherwise}\ . 
	\end{cases}
\end{align}

\begin{lemma}\label{lem:cohcompA}
	For any pair $1\leq i \leq j \leq k$ one has 
	\begin{align}\label{eq:chaincohA}
		H^0(\scrO_{C_{i,j}}) \simeq \C\ , \quad 
		H^p(\scrO_{C_{i,j}}) =0\text{ for }   p\geq 1\ ,
	\end{align}
	and
	\begin{align}\label{eq:chaincohB}
		H^1(\scrO_{C_{i,j}}(C_{i,j}))&\simeq \C\ ,\\
		H^p(\scrO_{C_{i,j}}(C_{i,j})) &=0\text{ for }
		p\geq 0, \ p \neq 1\ .
	\end{align}
	Moreover for any pair $1<i \leq j \leq k$ 
	\begin{align}\label{eq:chaincohC} 
		H^p(\scrO_{C_{i,j}}(C_{i-1,j})) =0 
	\end{align}
	for all $p\geq 0$.
\end{lemma} 
\begin{proof}
	The proof of equations \eqref{eq:chaincohA} and \eqref{eq:chaincohB} proceeds by induction on $j-i$. Both claims are clear for $j=i$ since $C_i \simeq \PP^1$ and $\scrO_{C_i}(C_i)\simeq \scrO_{C_i}(-2)$ for $i=1, \ldots, k$.  
	
	Suppose $i<j$. Then one has the canonical exact sequences 
	\begin{align}
		0\longrightarrow \scrO_{C_i}(-1) \longrightarrow \scrO_{C_{i,j}} \longrightarrow \scrO_{C_{i+1,j}} \longrightarrow 0\ , \\
		0\longrightarrow \scrO_{C_{i+1,j}}(-C_i) \longrightarrow \scrO_{C_{i,j}} \longrightarrow \scrO_{C_i} \longrightarrow 0\ .
	\end{align}
	The second sequence yields a third exact sequence 
	\begin{align}
		0\longrightarrow \scrO_{C_{i+1,j}}(C_{i+1,j}) \longrightarrow \scrO_{C_{i,j}}( C_{i,j}) \longrightarrow \scrO_{C_i}(-1) \longrightarrow 0
	\end{align}
	by taking a tensor product with the line bundle $\scrO_S( C_{i,j})$. Then equations \eqref{eq:chaincohA} and \eqref{eq:chaincohB} follow by an easy inductive argument using long exact sequences in cohomology. 
	
	The proof of \eqref{eq:chaincohC} is analogous. For $j=i$ one has $\scrO_{C_{i,j}}(C_{i-1,j}) \simeq \scrO_{C_i}(-1)$ hence the claim is obvious. For $j>i$ the inductive step uses the exact sequence 
	\begin{align}
		0\longrightarrow \scrO_{C_i}(-1) \longrightarrow \scrO_{C_{i,j}} \longrightarrow \scrO_{C_{i+1,j}}\longrightarrow 0\ .
	\end{align}
	Taking a tensor product with $\scrO_S(-C_{i-1,j})$ one obtains a second exact sequence 
	\begin{align}
		0\longrightarrow \scrO_{C_i}(-1) \longrightarrow \scrO_{C_{i,j}}(-C_{i-1,j})  \longrightarrow \scrO_{C_{i+1,j}}(-C_{i,j}) \longrightarrow 0\ ,
	\end{align}
	since $C_{i-1}$ and $C_{i+1}$ are disjoint. Then the claim follows again by induction. 
\end{proof}

Lemma~\ref {lem:cohcompA} yields:
\begin{corollary}\label{cor:cohcorA}
	For any $1\leq i \leq j \leq k$, one has 
	\begin{align}\label{eq:chaincohD}
		H^p(\scrO_S(-C_{i,j})) = 0\text{ for }p\geq 0
	\end{align}
	and
	\begin{align}\label{eq:chaincohE} 
		H^p(\scrO_S(C_{i,j})) \simeq \C\text{ for }0\leq p \leq 1\ .
	\end{align}
	Moreover,
	\begin{align}\label{eq:chaincohF} 
		H^2(\scrO_S(C_{i,j})) = 0\ .	
	\end{align}
\end{corollary} 
\begin{proof}
	All claims follow from Lemma~\ref {lem:cohcompA} using the canonical exact sequences 
	\begin{align}
		0 \longrightarrow \scrO_S(-D)\longrightarrow \scrO_S \longrightarrow \scrO_D\longrightarrow 0\ ,  \\
		0 \longrightarrow \scrO_S\longrightarrow \scrO_S(D) \longrightarrow \scrO_D(D)\longrightarrow 0 \ , 
	\end{align}
	associated to any effective divisor $D$ on $S$. One also has to use the fact that $S$ is rational, hence $H^p(\scrO_S) =0$ for $p\geq 1$.
\end{proof}

\providecommand{\bysame}{\leavevmode\hbox to3em{\hrulefill}\thinspace}
\providecommand{\MR}{\relax\ifhmode\unskip\space\fi MR }
\providecommand{\MRhref}[2]{%
	\href{http://www.ams.org/mathscinet-getitem?mr=#1}{#2}
}
\providecommand{\href}[2]{#2}


\begin{thebibliography}{BBHR09}
	
	\bibitem[BBHR09]{BBHR_FM}
	C.~Bartocci, U.~Bruzzo, and D.~Hern{\'a}ndez~Ruip{\'e}rez,
	\textit{Fourier-{M}ukai and {N}ahm transforms in geometry and mathematical
		physics}, Progress in {M}athematics {\bf 276}, Birkh{\"a}user, 2009.
	
	\bibitem[Bec94]{Beck_Quantum}
	J.~Beck, \textit{Braid group action and quantum affine algebras}, Comm. Math.
	Phys. \textbf{165} (1994), no.~3, 555--568.
	
	\bibitem[BKR01]{BKR}
	T.~Bridgeland, A.~King, and M.~Reid, \textit{The {M}c{K}ay correspondence as an
		equivalence of derived categories}, J. Amer. Math. Soc. \textbf{14} (2001),
	no.~3, 535--554.
	
	\bibitem[BCS20]{BCS_Critical_loci}
	T.~ Bozec, D.~Calaque, and S.~Scherotzke, \textit{Relative critical loci and quiver moduli}, \href{https://arxiv.org/pdf/2006.01069.pdf}{\sf arXiv:2006.01069}, 2020.
	
	\bibitem[BSV20]{BSV_Nilpotent}
	T.~Bozec, O.~Schiffmann, and \'{E}. Vasserot, \textit{On the number of points of
		nilpotent quiver varieties over finite fields}, Ann. Sci. \'{E}c. Norm.
	Sup\'{e}r. (4) \textbf{53} (2020), no.~6, 1501--1544.
	
	\bibitem[Con00]{Conrad_duality}
	B.~Conrad, \textit{Grothendieck duality and base change}, Lecture Notes in Mathematics, 1750. Springer-Verlag, Berlin, 2000. vi+296 pp.
	
	\bibitem[Cra10]{Cramer_Hall}
	T.~Cramer, \textit{Double {H}all algebras and derived equivalences}, Adv. Math.
	\textbf{224} (2010), no.~3, 1097--1120.
	
	\bibitem[CBH98]{CH_DeformedPreprojectiveAlgebras}
	W.~Crawley-Boevey and M.~P. Holland, \textit{Noncommutative deformations of
		{K}leinian singularities}, Duke Math. J. \textbf{92} (1998), no.~3.
	
	\bibitem[Dav16]{Davison_integrality}
	B.~Davison, \textit{The integrality conjecture and the cohomology of
		preprojective stacks}, \href{https://arxiv.org/abs/1602.02110}{\sf
		arXiv:1602.02110}, 2016.
	
	\bibitem[DM20]{Davison_DT}
	B.~Davison and S.~Meinhardt, \textit{Cohomological {D}onaldson-{T}homas theory
		of a quiver with potential and quantum enveloping algebras}, Invent. Math.
	\textbf{221} (2020), no.~3, 777--871.
	
	\bibitem[DPS22]{DPS_Torsion_Pairs}
	D.-E. Diaconescu, M.~Porta, and F.~Sala, \textit{Cohomological Hall algebras and their representations via torsion pairs}, \href{https://arxiv.org/abs/2207.08926}{\sf arXiv:2207.08926}, 2022.
	
	\bibitem[DLP85]{DLP_StableSheaves}
	J.-M. Drezet and J.~Le~Potier, \textit{Fibr\'{e}s stables et fibr\'{e}s
		exceptionnels sur {${\bf P}_2$}}, Ann. Sci. \'{E}cole Norm. Sup. (4)
	\textbf{18} (1985), no.~2, 193--243.
	
	\bibitem[DK19]{Dyckerhoff_Kapranov_Higher_Segal}
	T.~Dyckerhoff and M.~Kapranov, \textit{Higher {S}egal spaces}, Lecture Notes in
	Mathematics, vol. 2244, Springer, Cham, 2019.
	
	\bibitem[Efi20]{Efimov_Finiteness}
	A.~I.~Efimov, \textit{Homotopy finiteness of some DG categories from algebraic geometry}, J. Eur. Math. Soc. \textbf{22} (2020), no.~9, 2879--2942. 
	
	\bibitem[GKV95]{GKV_Langlands}
	V.~Ginzburg, M.~Kapranov, and \'{E}. Vasserot, \textit{Langlands reciprocity for
		algebraic surfaces}, Math. Res. Lett. \textbf{2} (1995), no.~2, 147--160.
	
	\bibitem[Gre95]{Green_Hall}
	J.~A. Green, \textit{Hall algebras, hereditary algebras and quantum groups},
	Invent. Math. \textbf{120} (1995), no.~2, 361--377.
	
	\bibitem[GSV83]{GSV_McKay}
	G.~Gonzalez-Sprinberg and J.-L. Verdier, \textit{Construction g\'{e}om\'{e}trique
		de la correspondance de {M}c{K}ay}, Ann. Sci. \'{E}cole Norm. Sup. (4)
	\textbf{16} (1983), no.~3, 409--449 (1984).
	
	\bibitem[Har77]{Hartshorne_AG}
	R.~Hartshorne, \textit{Algebraic geometry}, Graduate Texts in Mathematics, No.
	52, Springer-Verlag, New York-Heidelberg, 1977.
	
	\bibitem[HP19]{Tilting_chains}
	L.~Hille and D.~Ploog, \textit{Tilting chains of negative curves on rational
		surfaces}, Nagoya Math. J. \textbf{235} (2019), 26--41.
	
		\bibitem[HL10]{Huybrechts_Lehn_Moduli_2010}
	D.~Huybrechts and M.~Lehn, \textit{The geometry of moduli spaces of sheaves},
	second ed., Cambridge Mathematical Library, Cambridge University Press,
	Cambridge, 2010.
	
	\bibitem[IUU10]{IUU_Stability}
	A.~Ishii, K.~Ueda, and H.~Uehara, \textit{Stability conditions on $A_n$-singularities},
	J. Differential Geom. \textbf{84} (2010), no.~1, 87--126.
	
	\bibitem[KK17]{Noncomm_chains}
	M.~{Kalck} and J.~{Karmazyn}, \textit{{Noncommutative Kn{\"o}rrer type
			equivalences via noncommutative resolutions of singularities}},
	\href{https://arxiv.org/abs/1707.02836}{\sf arXiv:1707.02836}, 2017.
	
	\bibitem[Kap97]{Kapranov_Hall}
	M.~Kapranov, \textit{Eisenstein series and quantum affine algebras}, vol.~84,
	1997, Algebraic geometry, 7, pp.~1311--1360.
	
	\bibitem[KV00]{Kleinian_derived}
	M.~Kapranov and \'E. Vasserot, \textit{Kleinian singularities, derived categories
		and {H}all algebras}, Math. Ann. \textbf{316} (2000), no.~3, 565--576.
	
	\bibitem[KV19]{COHA_surface}
	\bysame, \textit{{The cohomological Hall algebra of a surface and factorization
			cohomology}}, \href{https://arxiv.org/abs/1901.07641}{\sf arXiv:1901.07641},
	2019.
	
	\bibitem[Kel18]{Keller_Deformed_Erratum}
	B.~Keller, \textit{Erratum to "Deformed Calabi-Yau completions"}, \href{https://arxiv.org/pdf/1809.01126.pdf}{\sf arXiv:1809.01126}, 2018.	
	
	\bibitem[KS11]{KS_Hall}
	M.~Kontsevich and Y.~Soibelman, \textit{Cohomological {H}all algebra, exponential
		{H}odge structures and motivic {D}onaldson-{T}homas invariants}, Commun.
	Number Theory Phys. \textbf{5} (2011), no.~2, 231--352.
	
	\bibitem[Moz11]{Motivic_DT_McKay}
	S.~Mozgovoy, \textit{{Motivic Donaldson-Thomas invariants and McKay
			correspondence}}, \href{https://arxiv.org/abs/1107.6044}{\sf
		arXiv:1107.6044}, 2011.
	
	\bibitem[Nag12]{N_derived}
	K.~Nagao, \textit{Derived categories of small toric {C}alabi-{Y}au 3-folds and
		curve counting invariants}, Q. J. Math. \textbf{63} (2012), no.~4, 965--1007.
	
	\bibitem[NN11]{NN}
	K.~Nagao and H.~Nakajima, \textit{Counting invariant of perverse coherent sheaves
		and its wall-crossing}, Int. Math. Res. Not. IMRN (2011), no.~17, 3885--3938.
	
	\bibitem[PS22]{Porta_Sala_Hall}
	M.~Porta and F.~Sala, \textit{Two-dimensional categorified {H}all algebras}, J. Eur. Math. Soc. (2022), DOI: 10.4171/JEMS/1303.
		
	\bibitem[RS17]{RS_Hall}
	J.~Ren and Y.~Soibelman, \textit{Cohomological {H}all algebras, semicanonical
		bases and {D}onaldson-{T}homas invariants for 2-dimensional {C}alabi-{Y}au
		categories (with an appendix by {B}.~{D}avison)}, Algebra, geometry, and
	physics in the 21st century, Progr. Math., vol. 324, Birkh\"{a}user/Springer,
	Cham, 2017, pp.~261--293.
	
	\bibitem[Rin90]{Ringel_Hall}
	C.~M. Ringel, \textit{Hall algebras and quantum groups}, Invent. Math.
	\textbf{101} (1990), no.~3, 583--591.

\bibitem[SS20]{Sala_Schiffmann}
F.~Sala and O.~Schiffmann, \textit{Cohomological {H}all algebra of {H}iggs
	sheaves on a curve}, Algebr. Geom. \textbf{7} (2020), no.~3, 346--376.

	\bibitem[Sch18]{Schiffmann_Kac}
	O.~Schiffmann, \textit{Kac polynomials and {L}ie algebras associated to quivers
		and curves}, Proceedings of the {I}nternational {C}ongress of
	{M}athematicians---{R}io de {J}aneiro 2018. {V}ol. {II}. {I}nvited lectures,
	World Sci. Publ., Hackensack, NJ, 2018, pp.~1393--1424.
	
	\bibitem[SV17]{SV_Yangians}
	O.~Schiffmann and \'E. Vasserot, \textit{On cohomological {H}all algebras of
		quivers: {Y}angians}, \href{https://arxiv.org/abs/1705.07491}{\sf
		arXiv:1705.07491}, 2017.
	
	\bibitem[SV20]{SV_generators}
	\bysame, \textit{On cohomological {H}all algebras of quivers: {G}enerators}, J.
	Reine Angew. Math. \textbf{760} (2020), 59--132.
	
	\bibitem[SY13]{Tilting_moduli}
	Y.~Sekiya and K.~Yamaura, \textit{Tilting theoretical approach to moduli spaces
		over preprojective algebras}, Algebr. Represent. Theory \textbf{16} (2013),
	no.~6, 1733--1786.
	
	\bibitem[{Sta}20]{stacks-project}
	The {Stacks project authors}, \textit{The stacks project},
	\url{https://stacks.math.columbia.edu}, 2020.
	
	\bibitem[TV07]{Toen_Vaquie_Moduli}
	B.~To{\"e}n and M.~Vaqui{\'e}, \textit{Moduli of objects in dg-categories},
	{A}nn. {S}ci. {\'E}cole {N}orm. {S}up. (4) \textbf{40} (2007), no.~3,
	387--444.
	
	\bibitem[Vak17]{Vakil_AG}
	R.~Vakil, \textit{The rising sea: Foundations of algebraic geometry}, 2017.
	
	\bibitem[VdB04]{VdB_Flops}
	M.~Van~den Bergh, \textit{Three-dimensional flops and noncommutative rings}, Duke
	Math. J. \textbf{122} (2004), no.~3, 423--455.
	
	\bibitem[VV20]{VV_Quantumloop}
	M.~Varagnolo and E.~Vasserot, \textit{K-theoretic {H}all algebras, quantum groups
		and super quantum groups}, \href{https://arxiv.org/abs/2011.01203}{\sf
		arXiv:2011.01203}, 2020.
	
	\bibitem[Wem11]{Wemyss_McKay}
	M.~Wemyss, \textit{The {${\rm GL}(2,\Bbb C)$} {M}c{K}ay correspondence}, Math.
	Ann. \textbf{350} (2011), no.~3, 631--659.
	
	\bibitem[Yos13]{YoshiokaPerverse}
	K.~Yoshioka, \textit{Perverse coherent sheaves and {F}ourier-{M}ukai transforms
		on surfaces, {I}}, Kyoto J. Math. \textbf{53} (2013), no.~2, 261--344.
	
	\bibitem[YZ20]{YZ_2Hall}
	Y.~Yang and G.~Zhao, \textit{On two cohomological {H}all algebras}, Proc. Roy.
	Soc. Edinburgh Sect. A \textbf{150} (2020), no.~3, 1581--1607.
	
\end{thebibliography}
\end{document}